\documentclass[12pt,reqno]{amsart}
\usepackage{amsmath,amsfonts,amsthm,amssymb,color,tikz}
\usepackage{mathrsfs}
\usepackage[T1]{fontenc}
\usepackage{enumerate}
\usepackage{hyperref}
\hypersetup{  
     colorlinks   = true,
     citecolor    = blue, 
     linkcolor    = blue  
}  

\usepackage{pdfsync}
\usepackage[font={scriptsize}]{caption}

\usepackage[left=1in, right=1in, top=1.1in,bottom=1.1in]{geometry}
\newcommand{\dd}{\mathrm{d}}

\newcommand{\dw}{\dot{W}}

\newcommand{\id}{\mbox{Id}}

\newcommand{\iot}{\int_{0}^{t}}

\newcommand{\ird}{\int_{\R^{d}}}

\newcommand{\ott}{[0,\tau]}

\newcommand{\R}{\mathbb R}
\newcommand{\N}{\mathbb N}

\newcommand{\Z}{\mathbb Z}

\newcommand{\be}{\mathbb{E}}

\newcommand{\bp}{\mathbf{P}}

\newcommand{\ca}{\mathcal A}

\newcommand{\cac}{\mathcal C}

\newcommand{\cd}{\mathcal D}
\newcommand{\ce}{\mathcal E}
\newcommand{\cf}{\mathcal F}
\newcommand{\cg}{\mathcal G}
\newcommand{\ch}{\mathcal H}
\newcommand{\ci}{\mathcal I}
\newcommand{\cj}{\mathcal J}

\newcommand{\cs}{\mathcal S}

\newcommand{\cu}{\mathcal U}

\newcommand{\al}{\alpha}

\newcommand{\ep}{\varepsilon}

\newcommand{\laa}{\Lambda}

\newcommand{\si}{\sigma}

\newcommand{\vp}{\varphi}

\newcommand{\lp}{\left(}
\newcommand{\rp}{\right)}
\newcommand{\lc}{\left[}
\newcommand{\rc}{\right]}

\newtheorem{theorem}{Theorem}[section]

\newtheorem{definition}[theorem]{Definition}

\newtheorem{hypothesis}[theorem]{Hypothesis}
\newtheorem{lemma}[theorem]{Lemma}
\newtheorem{notation}[theorem]{Notation}

\newtheorem{proposition}[theorem]{Proposition}

\theoremstyle{remark}
\newtheorem{remark}[theorem]{Remark}

\theoremstyle{remark}
\newtheorem{example}[theorem]{Example}

\newcommand{\bean}{\begin{eqnarray*}}
\newcommand{\eean}{\end{eqnarray*}}
\newcommand{\ben}{\begin{enumerate}}
\newcommand{\een}{\end{enumerate}}
\newcommand{\beq}{\begin{equation}}
\newcommand{\eeq}{\end{equation}}

\begin{document}

\title[Stochastic heat equation]
{On the   necessary and sufficient conditions 
  to solve a    heat equation 
 with general Additive   Gaussian noise}

\date{\today}

\author{Yaozhong Hu, Yanghui Liu,  Samy Tindel }
\address{Yaozhong Hu:  
Department  of Mathematical and 
 Statistical Sciences, University of Alberta,  Edmonton,   Canada, T6G 2G1} 

\address{Yanghui Liu,  Samy Tindel: Department of Mathematics, Purdue University, West Lafayette, IN 47907, USA}
\email{yaozhong@ualberta.ca\,,\ \ liu2048@purdue.edu\,,\ \ stindel@purdue.edu}


\thanks{S. Tindel is supported by the NSF grant  DMS1613163}
\date{}

\begin{abstract}
In this note we consider   stochastic heat equation with general additive Gaussian noise. Our aim is to derive   some necessary and sufficient conditions on the   Gaussian noise   in order to solve the corresponding heat equation. We   investigate this problem invoking two different methods, respectively based on variance computations and on path-wise considerations in Besov spaces.  We   are going to see that, as anticipated, both approaches lead to the same necessary and sufficient condition on the noise. In addition, the path-wise approach    brings out   regularity results for the solution.  
\end{abstract}

\maketitle

\section{Introduction}

In this paper, we are concerned with the following stochastic heat equation with additive noise:
\begin{equation}\label{eqn.anderson}
\begin{cases}
&\partial_{t} u    = \frac12 \Delta u   + \dw , \qquad t\in\ott,\, x\in\R^{d} \\
&u(0, x)   = 0,
\end{cases}
\end{equation}
where     $W$  is a general centered Gaussian field  with time covariance $R$ and spatial spectral measure   $\mu$ (see Definition \ref{def.gau} for further details), and where $\dot{W} = \frac{\partial^{d+1} W }{\partial t\partial x_{1}\cdots \partial x_{d}} $. In the recent years, there has been an active line of research aiming at a complete definition of stochastic heat equations driven by rough space-time noises. Among the numerous contributions in rough environments, let us highlight the following ones:

\noindent
\emph{(i)} The first efforts in this direction concern the definition of equation \eqref{eqn.anderson} driven by a Brownian motion $W$ in time. In this context, the articles \cite{Da, PZ} give (among other results) optimal conditions on the space covariance of $W$ so that the solution to \eqref{eqn.anderson} is function-valued.

\noindent
\emph{(ii)}  A lot of efforts have been devoted recently to the study of multiplicative stochastic heat equations driven by fractional noises in both space and time. A particular emphasis has been made on the effects of the noise on scaling exponents and asymptotic behavior of the solution. Among the abundant literature on this topic, let us mention the references~\cite{C, FK, HHNT}.

\noindent
\emph{(iii)} The acclaimed theory of regularity structures was introduced (see \cite{Hai, Hai1}) in order to solve highly nonlinear systems in rough environments, which require  renormalization techniques. Interestingly enough, this method is applied in \cite{Deya} to a family of stochastic heat equations similar to \eqref{eqn.anderson}.

The current paper can be seen as another step towards a better understanding of heat equations in rough environments. Namely, our  aim is to find optimal conditions in both space and time such that equation \eqref{eqn.anderson} admits a function (versus distribution)-valued solution. Otherwise stated, we wish to give necessary and sufficient conditions so that equation \eqref{eqn.anderson} can be defined without renormalization. Specifically, we will prove the following result (see Theorem \ref{thm1} for a more precise statement): 
\begin{theorem}\label{thm0}
Let $W$ be a centered Gaussian field with time covariance $R$ and spatial spectral measure $\mu$. 
Suppose that $R$ is a continuous function such that
\begin{eqnarray}\label{hyp.R}
   | R(t,t)+R(s,s')-R(s,t)-R(t, s') |  \asymp  |t-(s\wedge s')|^{\beta}, \quad   0\leq s,s'<t
\end{eqnarray}
for some $\beta\in (0, 2]$, where we write $a\asymp b$ if $a\leq K_{1} b$ and $b\leq K_{2}a$ for two  constants $K_{1} ,K_{2}>0$. 
Then equation \eqref{eqn.anderson} admits a random field solution $\{u(t, x); \, t\geq 0 , x\in \R^{d}\}$ if and only if   the following conditions is satisfied:
\begin{eqnarray}\label{hyp.mu}
  \int_{\R^{d}} \frac{1}{1+|\xi|^{2\beta}} \mu(\dd \xi)<\infty .
\end{eqnarray}
 \end{theorem}
\noindent
As the reader might see, our conditions \eqref{hyp.R} and \eqref{hyp.mu} are simple enough, while giving an if and only if condition of existence.

It is well-known that 
there are essentially two possible ways to consider equation \eqref{eqn.anderson}: the stochastic method based on Wiener integrals and the path-wise method relying on Young type integration. Interestingly, 
we have been able to prove Theorem \ref{thm0} resorting to both methods. Let us briefly explain our approaches:

\noindent
\emph{(1)}
The stochastic 
  method  is  based on the variance computations for Wiener integrals, involving the heat kernel and the covariance function of $W$. 
  The additional ingredients in our proof with respect to previous works (see e.g.   \cite{BT, TX})   are   some subtle partial integration by parts, which  are at the heart of the proof of Theorem \ref{thm0}.

\noindent
\emph{(2)}
The path-wise method resorts to the action of the heat semigroup on Besov spaces.  Once preliminary notions on harmonic analysis are given, it yields quite a simple solution to our problem and also brings some information about the regularity of the solution considered as a H\"older continuous function for free. However, we should stress the fact that the conditions obtained in this framework are only sufficient and slightly non optimal (namely condition~\eqref{hyp.mu} in Theorem \ref{thm0} is replaced by $\int_{\R^{d}} (1+|\xi|^{2\beta-\epsilon})^{-1} \mu(\dd \xi)<\infty$ for some $\epsilon>0$).
 
 Let us   mention that   we  have been focusing  on the additive Gaussian case in this paper in order to investigate the limits of our methods on a simple enough case. An interesting while   more demanding problem would be to handle the multiplicative case. This will be dealt with in some subsequent papers.  

Here is the organization of the paper. In the next section, we  set up some preliminary material on the Gaussian noises  we are dealing with. In Section \ref{section.solution}, we solve equation \eqref{eqn.anderson} by the stochastic method we have mentioned above. In Section \ref{section.solution2} we focus on path-wise techniques.

\paragraph{\textbf{Notations.}}
In the remainder of the article, all generic constants will be denoted by $K$, and their value 
may vary from different occurrences. We  denote by $p_{t}(x)$ the
$d$-dimensional heat kernel $p_{t}(x)=(2\pi
t)^{-d/2}e^{-|x|^2/2t} $, for any $t > 0$, $x \in \R^d$. $\N$ stands for the set of natural numbers: $\{0,1,2,\dots\}$.

\section{Noise model}\label{section.noise}

Let us start by introducing some basic notions on Fourier transforms
of functions: the space of real-valued infinitely differentiable
functions with compact support is denoted by $\mathcal{D} (
\mathbb{R}^d)$ or $\mathcal{D}$. The space of Schwartz functions is
denoted by $\mathcal{S} ( \mathbb{R}^d)$ or $\mathcal{S}$. Its dual,
the space of tempered distributions, is $\mathcal{S}' (
\mathbb{R}^d)$ or $\mathcal{S}'$. If $u$ is a vector  of tempered
distributions from  $\mathbb{R}^d$ to $\mathbb{R}^n$,  then we write
$u \in \mathcal{S}' (\mathbb{R}^d, \mathbb{R}^n)$. The Fourier
transform is defined with the normalization
\[ \mathcal{F}u ( \xi)  = \int_{\mathbb{R}^d} e^{- \imath \langle
   \xi, x \rangle} u ( x) \dd x, \]
so that the inverse Fourier transform is given by $\mathcal{F}^{- 1} u ( \xi)
= ( 2 \pi)^{- d} \mathcal{F}u ( - \xi)$.

 Let    $\mu $ be a non-negative   measure on $\R^{d}$. 
 The spatial covariance of our noise will be determined by a Hilbert space called $\ch$, defined as the completion  of $\cs(\R^{d})$ under the inner product:
 \begin{eqnarray}\label{eqn.cov.x}
\langle  \varphi, \psi  \rangle_{\ch} := \int_{\R^{d}} \cf(\varphi) (\xi) \overline{\cf(\psi)(\xi)}\mu(\dd\xi). 
\end{eqnarray}

As far as the time covariance of our noise is concerned, we shall consider   a continuous positive definite function $R$  on $\R_{+}^{2}$.
For convenience, 
the rectangular increments of $R$ will be denoted,
 for $s<t$ and $u<v$, by 
 \begin{eqnarray}\label{eqn.cov.t}
R 
\lp\begin{matrix}
s&t\\ u&v
\end{matrix}\rp
 = R(v, t) -R(v, s)-R(u,t)+R(u, s). 
\end{eqnarray}
We also denote by $\ce(\R_{+})$ the space of simple functions on $\R_{+}$.  
With those preliminary notations in hand, our noise is defined as follows: 
\begin{definition}\label{def.gau}
On a complete probability space
$(\Omega, \mathscr{F},\bp)$ we consider a Gaussian noise $W$ encoded by a
centered Gaussian family $\{W(h) ; \, h\in \ce(\R_{+})\times \cs(\R^{d})\}$, whose covariance structure
is given as follows: consider $g = \mathbf{1}_{[s,t]} \otimes \varphi $ and $h = \mathbf{1}_{[u,v]} \otimes \psi  $ with $ s<t $, $u<v$ and $ \varphi, \psi\in \cs(\R^{d}) $. Then we have
 \begin{eqnarray}\label{eqn.cov.w}
\be(W(g)W(h))  &=&  R 
\lp\begin{matrix}
s&t\\ u&v
\end{matrix}\rp
\int_{\R^{d}} \cf(\varphi) (\xi) \overline{\cf(\psi)(\xi)}\mu(\dd\xi),  
\end{eqnarray}
where $\cf \varphi$ refers to the Fourier transform with respect to the space variable only. 
\end{definition}

In this paper, due to the singularities of the heat kernel and of our covariance function, we will define our Wiener integrals via regularization. This is the contents of the definition below. 
   \begin{definition}\label{def.int}
Let $g $ be a measurable function on $ \R_{+}\times \R^{d}$ such that 
 $g(s,\cdot) \in \ch$, where $\ch$ is the Hilbert space defined by \eqref{eqn.cov.x}. For $\ep>0$ we set $t_{k}=t_{k}^{\ep}=k\ep$, $k\in \N$ and also $t_{\ep}= t-\sqrt[3]{\ep} $. We define
 \begin{eqnarray}\label{eqn.wiener}
   \int_{0}^{t}  \int_{\R^{d}} g(s, y)  W(\dd y, \dd s) 
    &=&
     \lim_{\ep\to 0+}  \sum_{0\leq t_{k}<t_{\ep}} \int_{t_{k}}^{t_{k+1}}\int_{\R^{d}} g(t_{k}, y) W(\dd y, \dd s)     ,
\end{eqnarray}
 whenever the $L^{2}(\Omega)$-limit of the right-hand side exists.  
Notice that the right-hand side of \eqref{eqn.wiener} is understood thanks to \eqref{eqn.cov.w}. 
  \end{definition}
  
\begin{remark}
Another natural way to introduce the Wiener integrals with respect to $W$,
which can be found in  the literature  (see e.g. \cite{HHNT, HLuN, HNS}), 
   is to consider the linear Gaussian space $\{W(h): h\in \cg\}$, where $\cg$ is the completion of $\ce\otimes \cs(\R^{d})$ with respect to the inner product \eqref{eqn.cov.w}. 
However, with the general definition of   Gaussian noise given  in Definition \ref{def.gau}, it is no longer convenient to investigate the Hilbert space $\cg$  if one wishes to solve an additive heat equation.    In Definition \ref{def.int}, we have adopted instead a Wong-Zakai type approximation similar to those suggested in \cite{GNRV, RV2} and not to ask for a complete understanding of the space $\cg$.  
\end{remark}


\section{Stochastic heat equation}\label{section.solution}


 Let $W$ be the Gaussian field introduced in Definition \ref{def.gau}. As mentioned in the introduction,   we are concerned with the heat equation \eqref{eqn.anderson} with additive noise on $\R^{d}$.
In this case it is well known (see e.g. \cite{DZ, TX}) that an explicit solution to \eqref{eqn.anderson}  should be given by the so-called stochastic convolution. Namely, for all $t>0$ and $x\in \R^{d}$, the solution $u$ to \eqref{eqn.anderson} is expressed as 
\begin{eqnarray}\label{eqn.solution}
u(t, x) &=&   \int_{0}^{t} \int_{\R^{n}} p_{t-s}(x-y)  W(\dd y, \dd s)
,
\end{eqnarray}
where $p_{t}$ stands for the heat kernel mentioned in the introduction, and where \eqref{eqn.solution} is 
 understood as a Wiener integral compatible with Definition  \ref{def.int}. 
 More generally, we will focus on the definition of a convolution of the form:
\begin{equation}\label{eq:def-convolution}
  \iot \ird g (s,y) \, W( \dd y, \dd s),
\end{equation}
for a deterministic kernel $g$.
 We wish to find optimal conditions on $R$ and $\mu$ such that expression \eqref{eqn.solution} makes sense.

\subsection{A discrete integration by parts formula}
In this section our computations rely on an   elementary   discrete integration by parts  formula.
For convenience, let us first introduce the following notation:
  \begin{notation}\label{notation.int}
  Consider a small constant 
   $\ep>0$, two positive numbers $s\geq 0$ and $t>0$, and set   $t_{k}=s+k\ep$ for $k\geq 0 $.  Let $f$ be a function on $\R_{+}$.  We 
   define a regularization of the   integral $\int_{0}^{t} f_{u}du$ in the following way:
 \begin{eqnarray*}
\int_{s}^{t } f_{u}\dd ^{\ep} u &=& \ep \sum_{s\leq t_{k}< t} f_{t_{k}} .
\end{eqnarray*}
 \end{notation} 
 Observe that for the discretized integral defined in   Notation \ref{notation.int}, the following elementary change of variables formula holds true:
 \begin{eqnarray}\label{eqn.notatione}
\int_{0}^{t} f_{s}\dd ^{\ep}s =  \int_{ \ep}^{t+ \ep} f_{s-\ep}\dd ^{\ep}s=  \int_{-\ep}^{t-\ep} f_{s+\ep}\dd ^{\ep}s. 
\end{eqnarray}

Let us now state our main technical tools, which is a partial integration by parts formula for the covariance function of $W$.

\begin{lemma}\label{lemma.ibp}
Let $t$, $\tilde{t}$ be two strictly positive numbers. Consider two   continuous functions
   $R$ and $\Gamma$     on   $[0,t] \times [0,\tilde{t}]$.   
For $\ep, \tilde{\ep}>0$, set 
\begin{eqnarray}\label{eqn.A}
 \ca(t, \tilde{t})&:=&
 \int_{0}^{\tilde{t}}\int_{0}^{t} \Gamma(s,s')
    R 
\lp\begin{matrix}
s'&s'+\tilde{\ep}\\ s &s +\ep
\end{matrix}\rp
      \dd ^{\ep}s\dd ^{\tilde{\ep}}s'  , \quad \quad t, \tilde{t}\geq 0, 
     \end{eqnarray}
     where the discretized integral is defined in Notation \ref{notation.int}.
Then   $\ca (t, \tilde{t})$ can be decomposed as follows:
 \begin{eqnarray}\label{eqn.aa}
\ca(t, \tilde{t})=\ca_{0}(t, \tilde{t}) +\ci_{0}(t, \tilde{t})+\ci_{1}(t, \tilde{t})+\ci_{2}(t, \tilde{t})+\ci_{3}(t, \tilde{t})+\ci_{4}(t, \tilde{t}),
\end{eqnarray}
where $\ca_{0}(t,\tilde{t})$ is the main term of an integration by parts:
\begin{align}
 \ca_{0}(t,\tilde{t})=&  \int_{0}^{\tilde{t}}\int_{0}^{t} 
    \Gamma \lp\begin{matrix}
s'-\tilde{\ep}&s' \\ s-\ep &s  
\end{matrix}\rp    R(s, s')  
      \dd ^{\ep}s\dd ^{\tilde{\ep}}s' ,
      \label{eqn.a0}
      \end{align}
      while $\ci_{0}$, \dots, $\ci_{4}$ are boundary terms defined respectively by: 
      \begin{align}
\ci_{0}(t,\tilde{t}) =&   \lp \int_{\tilde{t}}^{\tilde{t}+\tilde{\ep}}\int_{t}^{t+\ep}   -   \int_{0}^{\tilde{\ep}}\int_{0}^{\ep}    \rp \Gamma(s-\ep,s'-\tilde{\ep})
   R(s, s')  
      \dd ^{\ep}s\dd ^{\tilde{\ep}}s'
 \nonumber  
 \\
 :=& \ci_{01}(t,\tilde{t})-\ci_{00}(t,\tilde{t})  ,
 \label{eqn.i0}
  \\
     \ci_{1}(t,\tilde{t}) =&     \int_{\tilde{\ep}}^{\tilde{t}} \int_{t}^{t+\ep}
       ( \Gamma(s-\ep,s'-\tilde{\ep})-\Gamma(s-\ep,s'))
      R(s, s')  
      \dd ^{\ep}s\dd ^{\tilde{\ep}}s' 
   \nonumber
   \\
   & -   \int_{0}^{\tilde{\ep}} \int_{t}^{t+ {\ep}}
    \Gamma(s-\ep,s' )
      R(s, s')  
      \dd ^{\ep}s\dd ^{\tilde{\ep}}s'
   \nonumber   \\
      :=&\ci_{11}(t,\tilde{t})- \ci_{10}(t,\tilde{t}),
  \label{eqn.i1}    \\
  \ci_{2}(t,\tilde{t})=&        \int_{\tilde{t}}^{\tilde{t}+\tilde{\ep}}  \int_{\ep}^{t}
    (\Gamma(s-\ep,s'-\tilde{\ep}) - \Gamma(s, s'-\tilde{\ep}) )
      R(s, s')  
      \dd ^{\ep}s\dd ^{\tilde{\ep}}s'
     \nonumber \\
      & -     \int_{\tilde{t}}^{\tilde{t}+\tilde{\ep}}  \int_{0}^{\ep}
       \Gamma (s, s'-\tilde{\ep})R(s,s')\dd ^{\ep}s\dd ^{\tilde{\ep}}s'
       \nonumber   \\
      :=&\ci_{21}(t,\tilde{t})- \ci_{20}(t,\tilde{t}) ,
 \label{eqn.i2}     \\
      \ci_{3}(t,\tilde{t})=&    \int_{\tilde{\ep}}^{\tilde{t}} \int_{0}^{ \ep}
       (\Gamma(s-\ep, s') - \Gamma(s-\ep, s'-\tilde{\ep})) R(s,s')\dd ^{\ep}s\dd ^{\tilde{\ep}}s' 
      \nonumber \\
       &+   \int_{0}^{\tilde{\ep}} \int_{0}^{ \ep}
        \Gamma(s-\ep, s')R(s,s')\dd ^{\ep}s\dd ^{\tilde{\ep}}s',
    \label{eqn.i3}  
 \\
       \ci_{4}(t,\tilde{t})=&  \int_{0}^{\tilde{\ep}} \int_{\ep}^{ t}
       (\Gamma(s, s'-\tilde{\ep}) - \Gamma(s-\ep, s'-\tilde{\ep})) R(s,s')\dd ^{\ep}s\dd ^{\tilde{\ep}}s'
     \nonumber  \\
   &     +   \int_{0}^{\tilde{\ep}} \int_{0}^{ \ep}
       \Gamma(s, s'-\tilde{\ep})R(s,s')\dd ^{\ep}s\dd ^{\tilde{\ep}}s'  .
       \label{eqn.i4}
\end{align}

\end{lemma}

\begin{proof}
Start from the definition \eqref{eqn.A} of $\ca(t, \tilde{t})$, and recall our notation \eqref{eqn.cov.t} for the rectangular increments of $R$. Then a series of elementary   change of variables  resorting to \eqref{eqn.notatione} yields 
\begin{align}
 \ca  (t,\tilde{t}) =&    \int_{0}^{\tilde{t}} \int_{0}^{ t} \Gamma(s,s')
    \lp R(s+\ep, s'+\tilde{\ep})  -R(s+\ep, s') -R(s, s'+\tilde{\ep})+ R(s, s') \rp 
      \dd ^{\ep}s\dd ^{\tilde{\ep}}s'  
 \nonumber
      \\
      =&   \int_{\tilde{\ep}}^{\tilde{t}+\tilde{\ep}} \int_{\ep}^{ t+\ep}  \Gamma(s-\ep,s'-\tilde{\ep})
      R(s, s')  
      \dd ^{\ep}s\dd ^{\tilde{\ep}}s'  
       -  \int_{0}^{\tilde{t}} \int_{\ep}^{t+ \ep}   \Gamma(s-\ep,s')
      R(s, s')  
      \dd ^{\ep}s\dd ^{\tilde{\ep}}s'  
  \nonumber
      \\
      &\quad-   \int_{\tilde{\ep}}^{\tilde{t}+\tilde{\ep}} \int_{0}^{ t}   \Gamma(s,s'-\tilde{\ep})
      R(s, s')  
      \dd ^{\ep}s\dd ^{\tilde{\ep}}s'  
       +    \int_{0}^{\tilde{t}} \int_{0}^{ t} \Gamma(s, s')
      R(s, s')  
      \dd ^{\ep}s\dd ^{\tilde{\ep}}s'  
    .
    \nonumber
\end{align}
We now rearrange those terms by separating the interval $[0,t]$ from other intervals of length~$\ep$ and $\tilde{\ep}$. We get
\begin{eqnarray}\label{eqn.a}
\ca(t,\tilde{t}) &=&  
      \ca_{0} (t,\tilde{t})+ \ca_{1}(t,\tilde{t}) -(\ca_{21}(t,\tilde{t})-\ca_{22}(t,\tilde{t})) -(\ca_{31}(t,\tilde{t})-\ca_{32}(t,\tilde{t})),   
\end{eqnarray}

where $\ca_{0} (t,\tilde{t})$ is defined in \eqref{eqn.a0} and  
\begin{eqnarray*}
  \ca_{1} (t,\tilde{t})&=&  \lp \int_{\tilde{\ep}}^{\tilde{t}+\tilde{\ep}} \int_{\ep}^{ t+\ep} - \int_{0}^{\tilde{t}} \int_{0}^{ t} \rp \Gamma(s-\ep,s'-\tilde{\ep})
      R(s, s')  
      \dd ^{\ep}s\dd ^{\tilde{\ep}}s'  
      \\
        \ca_{21}(t,\tilde{t})-\ca_{22}(t,\tilde{t}) &=&  \lp\int_{0}^{ \tilde{t}} \int_{t}^{t+\ep}  - \int_{0}^{ \tilde{t}} \int_{0}^{ \ep}  \rp
       \Gamma(s-\ep,s')
      R(s, s')  
      \dd ^{\ep}s\dd ^{\tilde{\ep}}s' 
        \\
       \ca_{31}(t,\tilde{t})-\ca_{32}(t,\tilde{t}) &=&   \lp\int_{\tilde{t}}^{ \tilde{t}+\tilde{\ep}} \int_{0}^{t }  -\int_{0}^{\tilde{\ep}}\int_{0}^{ t}  \rp
       \Gamma(s, s'-\tilde{\ep})
      R(s, s')  
      \dd ^{\ep}s\dd ^{\tilde{\ep}}s'. 
\end{eqnarray*}
Next  we   decompose  the rectangles $[\ep, t+\ep]\times [\tilde{\ep}, \tilde{t}+\tilde{\ep}]  $ and $[0,t]\times[0,\tilde{t}]$ in order to get
\begin{align}\label{eqn.a1}
  \ca_{1}(t,\tilde{t})
     =&
       \bigg( \int_{\tilde{t}}^{\tilde{t}+\tilde{\ep}}\int_{t}^{t+\ep} +\int_{\tilde{t}}^{\tilde{t}+\tilde{\ep}}\int_{\ep}^{t}  +   \int_{\tilde{\ep}}^{\tilde{t}}\int_{t}^{t+\ep}   -  \int_{0}^{\tilde{\ep}}\int_{0}^{\ep}  
   \nonumber    \\
       & \quad\quad\quad\quad
       - \int_{0}^{\tilde{\ep}}\int_{\ep}^{t} -\int_{\tilde{\ep}}^{\tilde{t}}\int_{0}^{\ep}   \bigg) \Gamma(s-\ep,s'-\tilde{\ep})
      R(s, s')  
      \dd ^{\ep}s\dd ^{\tilde{\ep}}s'
 \nonumber
       \\
       :=&   \ca_{11}(t,\tilde{t})+ \ca_{12}(t,\tilde{t})+ \ca_{13}(t,\tilde{t})- \ca_{14}(t,\tilde{t})- \ca_{15}(t,\tilde{t})- \ca_{16}(t,\tilde{t}) \,.
\end{align}

Now substituting \eqref{eqn.a1} into \eqref{eqn.a} and rearranging the terms we obtain
\begin{eqnarray*}
\ca (t,\tilde{t})&=& \ca_{0}(t,\tilde{t})
+( \ca_{11}(t,\tilde{t})-\ca_{14}(t,\tilde{t}))+
( \ca_{13}(t,\tilde{t}) -\ca_{21}(t,\tilde{t}))
\\
&&+(\ca_{12}(t,\tilde{t})-\ca_{31}(t,\tilde{t}))+( \ca_{22}(t,\tilde{t})-\ca_{16}(t,\tilde{t}))+( \ca_{32}(t,\tilde{t})-\ca_{15}(t,\tilde{t})).
\end{eqnarray*}
The identity \eqref{eqn.aa} then follows  from 
  the expression of $\ci_{i}$  in \eqref{eqn.i0}-\eqref{eqn.i4} and  observing that 
\begin{eqnarray*}
\ci_{0}= \ca_{11}-\ca_{14}, \quad
\ci_{1}= \ca_{13} -\ca_{21} , \quad \ci_{2} = \ca_{12}-\ca_{31},  \quad \ci_{3}=\ca_{22}-\ca_{16}, \quad \ci_{4}=\ca_{32}-\ca_{15}.
\end{eqnarray*}
This completes the proof of the lemma.
  \end{proof}

\subsection{$L^{2}$ convergence for the stochastic heat equation}
We will now use the integration by parts formula stated in Lemma \ref{lemma.ibp} in order to estimate the Wiener integral defining our solution $u$. 
In the next result, we first calculate the $L^{2}(\Omega)$-norm of the Wong-Zakai type approximation in~\eqref{eqn.wiener}. 

\begin{lemma}\label{lemma.ue.L2}
Consider a small constant $\ep>0$, and $(t, x) \in \R_{+}\times \R^{d}$. Denote  $t_{\ep}=t-\sqrt[3]{\ep}$.  Let $h:[0, t)\to \R^{d} \to \R_{+}$ be defined by  $h  (s,y)=       p_{t-s} (x-y) $, where we recall that $p_{t}$ stands for the heat kernel on $\R^{d}$ (see Notation  at  the end of the introduction). We define $u_{\ep}(t, x)$ by:
\begin{eqnarray}\label{eqn.ue}
u_{\ep}(t, x) &=& \sum_{0\leq t_{k}<t_{\ep}} \int_{t_{k}}^{t_{k+1}}\int_{\R^{d}} h(t_{k}, y) W(\dd y, \dd s) .
\end{eqnarray}  
    Then the following holds true:
    
\noindent
\emph{(i)} The covariance between $u_{\ep}(t, x)$ and $u_{\tilde{\ep}}(t, x)$   can be expressed as:
 \begin{align}\label{eqn.ue.l2}
&
\be(u_{\ep}(t, x)u_{\tilde{\ep}}(t, x)) 
\nonumber
\\
&\quad= (\ep\tilde{\ep})^{-1}\lp
 \ca_{0}(t_{\ep}, t_{\tilde{\ep}}) +\ci_{0}(t_{\ep}, t_{\tilde{\ep}})+\ci_{1}(t_{\ep}, t_{\tilde{\ep}})+\ci_{2}(t_{\ep}, t_{\tilde{\ep}})+\ci_{3}(t_{\ep}, t_{\tilde{\ep}})+\ci_{4}(t_{\ep}, t_{\tilde{\ep}})
 \rp,
\end{align}
where  $ \ca_{0}$, $\ci_{0}$, $\ci_{1}$, $\ci_{2}$, $\ci_{3}$, $\ci_{4} $ are defined by relations \eqref{eqn.a0}-\eqref{eqn.i4},  with  a function $\Gamma$ defined on $[0 , \tau]^{2}$ by:   
\begin{eqnarray}\label{eqn.kk}
\Gamma(s,s') &=&   \int_{\R^{d}} e^{-\frac{(2t-s-s')|\xi|^{2}}{2}} \mu(d\xi) .
 \end{eqnarray}

\noindent
\emph{(ii)}
 Furthermore,  
 the kernel $\Gamma$ given by \eqref{eqn.kk} is differentiable on $[0,t)^{2}$ and for $0\leq s,s'<t$ we have 
\begin{eqnarray}
\frac{\partial \Gamma}{\partial s}  (s,s')=\frac{\partial \Gamma}{\partial s'}  (s,s') &=& \int_{\R^{d}}\frac{|\xi|^{2}}{2} e^{-\frac{(2t-s-s')|\xi|^{2}}{2}} \mu(\dd \xi) 
\label{eqn.dk}
\\
\frac{\partial^{2} \Gamma}{\partial s' \partial s}  (s, s') &=& \int_{\R^{d}} \frac{|\xi|^{4}}{4} e^{-\frac{(2t-s-s')|\xi|^{2}}{2}} \mu(\dd \xi) 
\label{eqn.ddk}
\\
\frac{\partial^{3} \Gamma}{\partial^{2} s' \partial s}  (s, s')=\frac{\partial^{3} \Gamma}{\partial  s' \partial^{2} s} (s, s') &=& \int_{\R^{d}} \frac{|\xi|^{6}}{8} e^{-\frac{(2t-s-s')|\xi|^{2}}{2}} \mu(\dd \xi) .
\label{eqn.dddk}
\end{eqnarray}
\end{lemma}
\begin{proof}
 Recall that $u_{\ep}$ is defined by \eqref{eqn.ue}. On each interval $[t_{k}, t_{k+1}]$ we use the elementary identity
 \begin{eqnarray*}
h(t_{k}, y) = \frac{1}{\ep} \int_{t_{k}}^{t_{k+1}} h(s, y) d^{\ep} s, 
\end{eqnarray*}
which is immediately seen from Notation \ref{notation.int}. Plugging this information into \eqref{eqn.ue} we get  
\begin{eqnarray*}
    u_{\ep}(t,x)
     &=&  \sum_{0\leq t_{k}<t_{\ep}} \int_{t_{k}}^{t_{k+1}}\int_{\R^{d}} h(t_{k}, y) W(\dd y, \dd s)  \\
   &=&
     \frac{1}{\ep} \int_{0}^{t_{\ep}} \dd ^{\ep}s \int_{s}^{s+\ep} \int_{\R^{d}} h(s, y)  W(\dd y, \dd r).
\end{eqnarray*}
Therefore, taking into account 
the covariance function of the Gaussian field $W$ in Definition~\ref{def.gau}, we obtain
  \begin{eqnarray*}
    \be \lc   u_{\ep}(t,x)u_{\tilde{\ep}}(t,x)   \rc 
    &=&  \frac{1}{\ep\tilde{\ep}} \int_{0}^{t_{\tilde{\ep}}}\int_{0}^{t_{\ep}} 
     \lc \tilde{\Gamma}(s,s')   R 
\lp\begin{matrix}
s'&s'+\tilde{\ep}\\ s &s +\ep
\end{matrix}\rp
      \rc
      \dd ^{\ep}s\dd ^{\tilde{\ep}}s' ,
\end{eqnarray*}
where the function $\tilde{\Gamma}$ is defined by
\begin{eqnarray}\label{eqn.tk}
\tilde{\Gamma}(s,s')   &=& \int_{\R^{d}} \cf( h)(\xi) \overline{\cf( h)(\xi)} \mu (\dd \xi).
\end{eqnarray}
Invoking Lemma     \ref{lemma.ibp}, 
  relation \eqref{eqn.ue.l2} is thus easily reduced to show that   $\tilde{\Gamma} = \Gamma$, where $\Gamma$ is given by   \eqref{eqn.kk}. 
 To this aim,  note that 
\begin{eqnarray}\label{eqn.cfp}
&&
\cf(p_{t-s}(x-y)) = e^{-\frac{(t-s)|\xi|^{2}}{2}} e^{-i\xi\cdot x}  
 . 
\end{eqnarray}
Substituting \eqref{eqn.cfp} into \eqref{eqn.tk} we immediately have $\tilde{\Gamma}=\Gamma$, which finishes the proof of \eqref{eqn.ue.l2}. 

The identities \eqref{eqn.dk}-\eqref{eqn.dddk} follows easily by noticing that 
\begin{eqnarray*}
 \cf \lp \frac{\partial }{\partial s}  p_{t-s}(x-y)\rp  =\frac{\partial }{\partial s} \cf(p_{t-s}(y-x))  = \frac{|\xi|^{2}}{2} e^{-\frac{(t-s)|\xi|^{2}}{2}} e^{-i\xi\cdot x} ,
\end{eqnarray*}
and the fact that $ e^{-\frac{(t-s)|\xi|^{2}}{2}}$ is  an increasing function of $s$ in order to apply monotone convergence.  
\end{proof}

With Lemma \ref{lemma.ue.L2} in hand, we will now bound the terms in \eqref{eqn.ue.l2} (see also \eqref{eqn.aa} for more precise definitions) individually. Let us start by  analyzing the terms 
 $\ci_{3}$ and $\ci_{4}$ in \eqref{eqn.ue.l2}.  
\begin{lemma}\label{lemma.i3i4}
Let   $\Gamma$ be
defined by \eqref{eqn.kk} and $R$ be a continuous function satisfying \eqref{hyp.R}.  
Let $\ci_{3}$ and $\ci_{4} $ be given by 
 \eqref{eqn.ue.l2} (see also \eqref{eqn.i3} and \eqref{eqn.i4}). 
 Then the following convergence holds true
\begin{eqnarray*}
\lim_{\ep, \tilde{\ep}\to0+}
\ep^{-1}\tilde{\ep}^{-1} \ci_{3} (t_{\ep}, t_{\tilde{\ep}})&=& \int_{0}^{t} \frac{\partial \Gamma}{\partial s'}(0, s') R(0, s')\dd s' + \Gamma(0 , 0) R(0, 0),
\\
\lim_{\ep, \tilde{\ep}\to0+}
\ep^{-1}\tilde{\ep}^{-1} \ci_{4} (t_{\ep}, t_{\tilde{\ep}}) &=& \int_{0}^{t} \frac{\partial \Gamma}{\partial s}(s, 0) R(s, 0)\dd s' + \Gamma(0 , 0) R(0, 0).
\end{eqnarray*}

\end{lemma}
\begin{proof} The convergences follow immediately from the continuity of $R$ and the monotonicity of $\Gamma$ and its derivatives and an application of the dominated convergence theorem.  Notice that when one integrates $\Gamma$ defined by \eqref{eqn.kk} there is a singularity at $s=s'=t$. However, this singularity is avoided for the terms $\ci_{3}$ and $\ci_{4}$, as well as for the terms \eqref{eqn.dk}-\eqref{eqn.dddk}.
\end{proof}

In order to proceed with our estimates, 
we now state an elementary lemma giving an upper bound on the increments of the function $R$. 

\begin{lemma}\label{lemma.R.incr}
Let $R$ be a covariance function satisfying relation \eqref{hyp.R} with $\beta>0$. Then for $u,v, t \in [0,\tau]$ we have 
\begin{eqnarray*}
|R(t,u) -R(t,v) | &\leq& K|u-v|^{\beta/2},  
\end{eqnarray*}
where $K$ is a constant depending on $\tau$. 
\end{lemma}
\begin{proof}
Let $X$ be a Gaussian process on $[0,\tau]$ with covariance function $R$ and $X_{0}=0$. Then by relation \eqref{hyp.R} we have
\begin{eqnarray}\label{eqn.cov.X}
|\be[(X_{t} -X_{u})(X_{t}-X_{v})] |= | R(t,t)+R(u,v)-R(u,t)-R(t, v) |  \leq  |t-(u\wedge v)|^{\beta}. 
\end{eqnarray}
The lemma then follows from the following relations
\begin{eqnarray*}
|R(t,u) -R(t,v) | = | \be[ X_{t}(X_{u}-X_{v})] | \leq \be[|X_{t}|^{2}]^{1/2} \be [|X_{u}-X_{v}|^{2}]^{1/2} \leq R(t,t)^{1/2}|u-v|^{\beta/2}, 
\end{eqnarray*}
where in the last inequality we have used relation \eqref{eqn.cov.X} with $u=v$.
\end{proof}

In order to handle the term $\ca_{0}(t_{\ep}, t_{\tilde{\ep}})$ in relations \eqref{eqn.a0} and \eqref{eqn.ue.l2} we will artificially introduce some rectangular increments. The lemma below takes care of the convergence of the rectangular increment $\tilde{\ca}_{0}^{\ep, \tilde{\ep}}$ derived from $\ca_{0}(t_{\ep}, t_{\tilde{\ep}})$ (specifically, we replace the term $R(s,s')$ in~\eqref{eqn.a0} by its rectangular increment).  
\begin{lemma}\label{lemma.a0}
Let   $\Gamma$ be defined by \eqref{eqn.kk} and $R$ be a continuous function  satisfying relation~\eqref{hyp.R}.
For $\ep, \tilde{\ep}>0$, we set 
\begin{eqnarray}\label{eqn.a0.tilde}
\tilde{\ca}_{0}^{\ep, \tilde{\ep}}
 = \int_{0}^{t_{\tilde{\ep}}}\int_{0}^{t_{\ep}}  
    \Gamma \lp\begin{matrix}
s'-\tilde{\ep}&s' \\ s-\ep &s  
\end{matrix}\rp 
R \lp\begin{matrix}
s&t \\ s' &t  
\end{matrix}\rp 
      \dd ^{\ep}s\dd ^{\tilde{\ep}}s'  .
\end{eqnarray}
Suppose that the spectral measure $\mu$ satisfies relation \eqref{hyp.mu}.  Then
\begin{eqnarray}\label{eqn.ca.lim}
\lim_{\ep, \tilde{\ep}\to0+}
\ep^{-1}\tilde{\ep}^{-1}
 \tilde{\ca}_{0}^{\ep, \tilde{\ep}} = \int_{0}^{t}\int_{0}^{t} \frac{\partial^{2}\Gamma}{\partial s\partial s'} (s,s') R \lp\begin{matrix}
s&t \\ s' &t  
\end{matrix}\rp 
    \dd s\dd s'   :=\cj.
 \end{eqnarray}

\end{lemma}

\begin{proof}
We first show that   the right-hand side of \eqref{eqn.ca.lim} is integrable. Note first that by relation~\eqref{hyp.R} and taking into account \eqref{eqn.ddk}, we have
\begin{eqnarray}\label{eqn.j.decomp}
\cj &\leq&
 \int_{0}^{t}\int_{0}^{t} \frac{\partial^{2}\Gamma}{\partial s\partial s'} (s,s') \left|R \lp\begin{matrix}
s&t \\ s' &t  
\end{matrix}\rp \right|
    \dd s\dd s'
      \leq  
     \int_{0}^{t}\int_{0}^{t} \frac{\partial^{2}\Gamma}{\partial s\partial s'} (s,s') |t-(s\wedge s')|^{\beta}
    \dd s\dd s'
  \nonumber  \\
    &\leq&  
     \int_{0}^{t}\int_{0}^{t} \int_{\R^{d}} \frac{|\xi|^{4}}{4} e^{-\frac{(2t-s-s')|\xi|^{2}}{2}} \mu(\dd \xi) |t-(s\wedge s')|^{\beta}
    \dd s\dd s'.
 \nonumber   
\end{eqnarray}
We can now split our estimate by writing 
\begin{eqnarray}\label{eqn.jdecomp}
\cj&\leq& \cj_{1} + \cj_{2}
\end{eqnarray}
where the terms $\cj_{1}$ and $\cj_{2}$ are defined by 
\begin{eqnarray*}
 \cj_{1}&=&  
      \int_{|\xi|>1}  \int_{0}^{t}\int_{0}^{t} \frac{|\xi|^{4}}{4} e^{-\frac{(2t-s-s')|\xi|^{2}}{2}} |t-(s\wedge s')|^{\beta}
    \dd s\dd s'  \mu(\dd \xi),
\\
\cj_{2} &=&   \int_{|\xi|\leq 1}    \int_{0}^{t}\int_{0}^{t} \frac{|\xi|^{4}}{4} e^{-\frac{(2t-s-s')|\xi|^{2}}{2}} |t-(s\wedge s')|^{\beta}
    \dd s\dd s'  \mu(\dd \xi).
\end{eqnarray*}

In the following, we bound $\cj_{1}$ and $\cj_{2}$ separately, starting with $\cj_{1}$. Indeed, 
since the   integral defining $\cj_{1}$ is symmetric in $s$ and $s'$, we have
\begin{eqnarray}\label{eqn.j1.bd}
 \cj_{1} &=&   2   \int_{|\xi|>1}  \int_{0}^{t}\int_{0}^{s} \frac{|\xi|^{4}}{4} e^{-\frac{(2t-s-s')|\xi|^{2}}{2}} |t-(s\wedge s') |^{\beta}
    \dd s'\dd s  \mu(\dd \xi)
 \nonumber
    \\
    &=&  2  \int_{|\xi|>1}  \frac{|\xi|^{4}}{4} \int_{0}^{t} e^{-\frac{(t-s)|\xi|^{2}}{2}} \int_{0}^{s}  e^{-\frac{( t-s')|\xi|^{2}}{2}} |t-  s'|^{\beta}
    \dd s'\dd s  \mu(\dd \xi)
  .
\end{eqnarray}
Next, we bound the inner integral in the right-hand side above as follows:
\begin{eqnarray*}
\int_{0}^{s}  e^{-\frac{( t-s')|\xi|^{2}}{2}} |t-  s'|^{\beta}
    \dd s' &\leq& \int_{0}^{\infty} r^{\beta} e^{-\frac{r|\xi|^{2}}{2}} \dd r = 2^{1+\beta} \Gamma(\beta+1) |\xi|^{-(2+2\beta)},
\end{eqnarray*}
where the last identity is obtained by a straightforward change of variable. 
In the same way, it is also readily checked that 
\begin{eqnarray*}
\int_{0}^{t} e^{-\frac{(t-s)|\xi|^{2}}{2}} \dd s &\leq& |\xi|^{-2}. 
\end{eqnarray*}
Plugging those two elementary estimates into \eqref{eqn.j1.bd}, we end up with:
 \begin{eqnarray}\label{eqn.cj1.estimate}
  \cj_{1} &\leq&  K   \int_{|\xi|>1} {|\xi|^{4}}  \frac{1}{ |\xi|^{2\beta+4}}   \mu(\dd \xi) =K   \int_{|\xi|>1}    \frac{1}{ |\xi|^{2\beta }}   \mu(\dd \xi)  <\infty,
\end{eqnarray}
where the last inequality follows from relation \eqref{hyp.mu}.

Let us now turn to an upper bound for $\cj_{2}$ defined in \eqref{eqn.jdecomp}. 
When $|\xi|\leq 1$, 
we have
\begin{eqnarray*}
\frac{|\xi|^{4}}{4} e^{-\frac{(2t-s-s')|\xi|^{2}}{2}} &\leq& 1.
\end{eqnarray*}
Hence $\cj_{2}$ is easily bounded as follows:
\begin{eqnarray}\label{eqn.j2.bdd}
\cj_{2}&\leq & \mu( \xi: |\xi|\leq 1) \int_{0}^{t} \int_{0}^{t} |t-(s \wedge s')|^{\beta} \dd s\dd s' <\infty. 
\end{eqnarray}
Applying the estimate of $\cj_{1}$ and $\cj_{2}$
obtained 
 in \eqref{eqn.cj1.estimate} and  \eqref{eqn.j2.bdd} to relation \eqref{eqn.j.decomp},    the integrability of  the right-hand side of \eqref{eqn.ca.lim} is trivially satisfied.

Next, by  \eqref{eqn.ddk} and the mean value theorem we have
\begin{eqnarray*}
0\leq  \frac{1}{\ep\tilde{\ep}} \Gamma \lp\begin{matrix}
s'-\tilde{\ep}&s' \\ s-\ep &s  
\end{matrix}\rp  = \frac{\partial^{2}\Gamma}{\partial s\partial s'} (s-\ep \theta, s'-\tilde{\ep} \theta') \leq \frac{\partial^{2}\Gamma}{\partial s\partial s'} (s, s') 
\end{eqnarray*}
uniformly in $\ep$ and $\tilde{\ep}$, 
where $\theta$ and $\theta'$ are some numbers between $0$ and $1$, and where the last inequality is due to the fact that the exponential function is monotone. Therefore,  
$\tilde{\ca}_{0}^{\ep, \tilde{\ep}}$   is dominated by  $\cj$.
By the dominated convergence theorem and by taking limits $\ep, \tilde{\ep}\to0+$ for $\tilde{\ca}_{0}^{\ep, \tilde{\ep}}$ we obtain the desired convergence \eqref{eqn.ca.lim}.
\end{proof}

In order to get the convergence of the boundary terms $\ci_{01}$, $\ci_{11}$, and $\ci_{21}$ in \eqref{eqn.i0}-\eqref{eqn.i2}, we need the following technical lemma: 
\begin{lemma}\label{lemma.g1g2}
Denote $t_{\ep}=t-\sqrt[3]{\ep}$ for   $\ep>0$, and let $R$ and $\Gamma$ be as in Lemma \ref{lemma.a0}. 
For $i=1,2,3,4$ we define small constants $\delta_{i}$ and some instants $t_{i} $ in the following way: for $i=1,3$ we take $\delta_{i}\in [0,\ep]$ and $t_{i}$ such that   $t_{i}: 0\leq t_{\ep}-t_{i}\leq \ep $, while for $i=2,4$ we consider $\delta_{i} \in [0, \tilde{\ep}]$ and $t_{i}$ such that   $t_{i}: 0\leq t_{\tilde{\ep}}-t_{i}\leq \tilde{\ep} $. 
On $[0,t]$, define $8$ piecewise continuous functions
\begin{eqnarray*}
u_{i}= u_{i}(u), \quad v_{i} =  v_{i} (v), \quad \quad \text{for }  i=1,2,3,4,
\end{eqnarray*}
satisfying 
 $ |u_{2}-u_{4}| \leq \ep$, $ |v_{2}-v_{4}| \leq \tilde{\ep}$, $0\leq \id-u_{i}\leq \ep$ and $0\leq \id-v_{i}\leq \tilde{\ep}$ for $i=1,3$. Besides, assume that $|u_{i}(u) - u_{i}(u+\ep) |\leq \ep $ and $|v_{i}(u) - v_{i}(u+\tilde{\ep}) |\leq \tilde{\ep} $ for $i=1, 2,3,4$. Denote $g=  \frac{\partial^{2}\Gamma}{\partial s\partial s'}$. 
 Set  
\begin{eqnarray*}
\cg_{1}=\int_{\delta_{1}}^{t_{1}}\int_{\delta_{2}}^{t_{2}} g (u_{1}, v_{1}) R(u_{2}, v_{2})\dd u\dd v  
\quad\quad\text{and}\quad\quad
\cg_{2}=\int_{\delta_{3}}^{t_{3}}\int_{\delta_{4}}^{t_{4}} g (u_{3}, v_{3}) R(u_{4}, v_{4})\dd u\dd v   .
\end{eqnarray*}	
Then the following convergence holds true
\begin{eqnarray*}
\lim_{\ep, \tilde{\ep}\to0+} | \cg_{1}-\cg_{2} | = 0. 
\end{eqnarray*}

\end{lemma}
\begin{proof}
Choose $\ep_{1}$ and $\ep_{2} $ such that $\ep_{1}+t_{3}=t_{1}$ and $\ep_{2}+t_{4}=t_{2}$.  
In the integral defining $\cg_{2}$, set $u = u+\ep_{1}$ and $v= v+\ep_{2}$. Setting $u_{j}' = u_{j} (u-\ep_{1})$ and $v_{j}' = v_{j} (v-\ep_{1})$, we get 
\begin{eqnarray*}
\cg_{2} &=&   \int_{\delta_{3}+\ep_{1}}^{t_{3}+\ep_{1}} \int_{\delta_{4}+\ep_{2}}^{t_{4}+\ep_{2}} g(u_{3}',v_{3}' ) R(u_{4}',v_{4}' ) \dd u\dd v
\\
&=&   \int_{\delta_{3}+\ep_{1}}^{t_{1}} \int_{\delta_{4}+\ep_{2}}^{t_{2}} g(u_{3}',v_{3}' ) R(u_{4}',v_{4}' ) \dd u\dd v
.
\end{eqnarray*}
We now introduce a slight   modification of $\cg_{2}$ called $\cg_{2}'$:
\begin{eqnarray*}
\cg_{2}'=\int_{\delta_{1}}^{t_{1}}\int_{\delta_{2}}^{t_{2}} g (u_{3}', v_{3}') R(u_{4}', v_{4}')\dd u\dd v   .
\end{eqnarray*}
We are going to show that
\begin{eqnarray}\label{eqn.g1g2p}
\cg_{1}-\cg_{2}' \to 0 \quad\text{and} \quad\cg_{2}-\cg_{2}' \to 0   \quad\quad\text{as \ }  \ep, \tilde{\ep}\to0+,
\end{eqnarray}
 which then concludes the proof.

Write
\begin{eqnarray}\label{eqn.g.diff}
 \cg_{1}-\cg_{2}'   &=&  \int_{\delta_{1}}^{t_{1}}\int_{\delta_{2}}^{t_{2}} g (u_{1}, v_{1}) ( R(u_{2}, v_{2}) - R(u_{4}', v_{4}') )\dd u\dd v   
 \nonumber
 \\
 &&+  \int_{\delta_{1}}^{t_{1}}\int_{\delta_{2}}^{t_{2}} (g (u_{1}, v_{1}) - g (u_{3}', v_{3}') ) R(u_{4}', v_{4}')\dd u\dd v
 \\
 &&
  := 
  \int_{\delta_{1}}^{t_{1}}\int_{\delta_{2}}^{t_{2}} \cg_{11} \dd u\dd v +  \int_{\delta_{1}}^{t_{1}}\int_{\delta_{2}}^{t_{2}}\cg_{12} \dd u\dd v
  \nonumber
\end{eqnarray}
Let us briefly show how to bound the term $\cg_{11}$. We first write 
\begin{eqnarray}\label{eqn.Rincr}
|R(u_{2}, v_{2}) - R(u_{4}', v_{4}') |  \leq | R(u_{2}, v_{2})-R(u_{4}', v_{2}) | + | R(u_{4}', v_{2})- R(u_{4}', v_{4}') |.
\end{eqnarray}
Owing to the fact that $|u_{2} - u_{4}'|\leq 2\ep$ and $|v_{2} - v_{4}'|\leq 2\tilde{\ep}$ we can invoke Lemma   \ref{lemma.R.incr} and our decomposition \eqref{eqn.Rincr} to get
\begin{eqnarray}\label{eqn.R.incr}
| R (u_{2}, v_{2})-R (u_{4}', v_{4}') | &\leq& K (\ep^{\beta/2}+\tilde{\ep}^{\beta/2} ) . 
\end{eqnarray}
Moreover, according to our standing assumptions the variable $u$ in $\cg_{11} $ satisfies 
\begin{eqnarray*}
u\leq t_{1}\leq t_{\ep} \leq t-\ep^{1/3},
\end{eqnarray*}
and similarly $v\leq t-\tilde{\ep}^{1/3}$. Therefore, 
 \begin{eqnarray}\label{eqn.R.incr2}
  K (\ep^{\beta/2}+\tilde{\ep}^{\beta/2} ) \leq K  (t-(u\wedge v))^{3\beta/2}, 
\end{eqnarray}
 and plugging \eqref{eqn.R.incr2} into \eqref{eqn.R.incr} we get 
 \begin{eqnarray}\label{eqn.R.incr3}
| R (u_{2}, v_{2})-R (u_{4}', v_{4}') | &\leq&   K  (t-(u\wedge v))^{3\beta/2}. 
\end{eqnarray}
Reporting this information into the expression of $\cg_{11}$ in \eqref{eqn.g.diff}, and taking into account the fact that  $g = \frac{\partial^{2} \Gamma}{\partial s\partial s'}$ satisfies  \eqref{eqn.ddk} we thus get
 \begin{eqnarray}\label{eqn.dominate.g1}
|\cg_{11}| \leq    (t-(u\wedge v))^{\frac{3\beta}{2} }
 \int_{\R^{d}}    \frac{|\xi|^{4}}{4} e^{-\frac{(2t-u-v)|\xi|^{2}}{2}}  \mu(\dd \xi) 
  .
\end{eqnarray}
  Similarly we also have
  \begin{eqnarray}\label{eqn.dominate.g2}
  |\cg_{12}| \leq  
 (t-(u\wedge v))^{3 }\int_{\R^{d}} \frac{|\xi|^{6}}{8} e^{-\frac{(2t-u-v)|\xi|^{2}}{2}} \mu(\dd \xi) .
\end{eqnarray}


In a similar way as in \eqref{eqn.j.decomp} it follows that both  functions in the right-hand side of \eqref{eqn.dominate.g1} and \eqref{eqn.dominate.g1} are   integrable, and thus by the dominated convergence   theorem
we are able to  pass the limit $\ep  ,\tilde{\ep}\to0+$ inside the integrals of \eqref{eqn.g.diff}. This yields the  convergence   $\cg_{1}-\cg_{2}' \to 0$.
Observe that a cubic power of $(t-(u\wedge v))$ is necessary in order to balance the term $|\xi|^{6}$ in the right-hand side of \eqref{eqn.dominate.g2}. This is the reason why we have imposed the condition $t_{\ep} = t-\ep^{1/3}$. 


Summarizing our considerations so far, we have proved that  $\cg_{1}-\cg_{2}' \to 0$ in \eqref{eqn.g1g2p}. In order to see that  $\cg_{2}-\cg_{2}'$, we write  
\begin{eqnarray*}
\cg_{2}-\cg_{2}'  &=&
 \lp
 \int_{\delta_{3}}^{t_{1}} \int_{\delta_{4}}^{t_{2}}-\int_{\delta_{3}-\ep_{1}}^{t_{1}} \int_{\delta_{4}-\ep_{2}}^{t_{2}}\rp g(u_{3}',v_{3}' ) R(u_{4}',v_{4}' ) \dd u\dd v
 \\
 &=&
  \lp\int_{\delta_{3}}^{\delta_{3}-\ep_{1}} \int_{\delta_{4}}^{t_{2}}+
  \int_{\delta_{3}-\ep_{1}}^{t_{1}} \int^{\delta_{4}-\ep_{2}}_{\delta_{4}}\rp g(u_{3}',v_{3}' ) R(u_{4}',v_{4}' ) \dd u\dd v.
\end{eqnarray*}
Since $g$ and $R$ are both continuous on $[-2\ep, 2\ep]\times [0,t]$ and $ [0,t] \times [-2\tilde{\ep}, 2\tilde{\ep}]$, we  immediately have    $\cg_{2}-\cg_{2}'\to 0$ as $\ep, \tilde{\ep}\to0+$. 
In conclusion,   we have found that relation \eqref{eqn.g1g2p} holds true. Therefore, it is readily checked that $\lim_{\ep, \tilde{\ep}\to 0+} | \cg_{1} - \cg_{2} |=0 $.
The proof is complete.
\end{proof}


With Lemma \ref{lemma.g1g2} in hand, we can now estimate some of the boundary terms derived from $\ca_{0} (t_{\ep}, t_{\tilde{\ep}})$ in Lemma \ref{lemma.ue.L2}. 
\begin{lemma}\label{lemma.a00}
Let the assumptions be as in Lemma \ref{lemma.a0} and suppose that relations \eqref{hyp.R} and~\eqref{hyp.mu} hold true. Set
\begin{eqnarray*}
&&\ca_{00} (t_{\ep}, t_{\tilde{\ep}})= \int_{0}^{t_{\tilde{\ep}}}\int_{0}^{t_{\ep}}   
    \Gamma \lp\begin{matrix}
s'-\tilde{\ep}&s' \\ s-\ep &s  
\end{matrix}\rp 
R (t,t) 
      \dd ^{\ep}s\dd ^{\tilde{\ep}}s',
     \\
     &&
      \ca_{01} (t_{\ep}, t_{\tilde{\ep}}) = \int_{0}^{t_{\tilde{\ep}}}\int_{0}^{t_{\ep}}  
    \Gamma \lp\begin{matrix}
s'-\tilde{\ep}&s' \\ s-\ep &s  
\end{matrix}\rp 
R (s,t) 
      \dd ^{\ep}s\dd ^{\tilde{\ep}}s',
      \\
   &&
      \ca_{02}  (t_{\ep}, t_{\tilde{\ep}})= \int_{0}^{t_{\tilde{\ep}}}\int_{0}^{t_{\ep}}   
    \Gamma \lp\begin{matrix}
s'-\tilde{\ep}&s' \\ s-\ep &s  
\end{matrix}\rp 
R (t,s') 
      \dd ^{\ep}s\dd ^{\tilde{\ep}}s' . 
\end{eqnarray*}
 Then 
\begin{eqnarray}
\lim_{\ep, \tilde{\ep}\to 0+} (\ep\tilde{\ep})^{-1} ( \ca_{00} (t_{\ep}, t_{\tilde{\ep}}) - \ci_{01} (t_{\ep}, t_{\tilde{\ep}})) &=& R(t,t) ( \Gamma(0, 0)-\Gamma(0,t)-\Gamma(t,0) ),
\label{eqn.a00.lim}
\\
\lim_{\ep,\tilde{\ep}\to 0+} (\ep\tilde{\ep})^{-1} ( -\ca_{01} (t_{\ep}, t_{\tilde{\ep}}) - \ci_{21} (t_{\ep}, t_{\tilde{\ep}})) &=& \int_{0}^{t} R(s,t) \frac{\partial \Gamma}{\partial s}(s,0)\dd s ,
\label{eqn.a01.lim}
\\
\lim_{\ep,\tilde{\ep}\to 0+} (\ep\tilde{\ep})^{-1} ( -\ca_{02} (t_{\ep}, t_{\tilde{\ep}}) - \ci_{11}  (t_{\ep}, t_{\tilde{\ep}})) &=&   \int_{0}^{t} R(t,s') \frac{\partial \Gamma}{\partial s'}(0,s') \dd s' ,
\label{eqn.a02.lim}
\end{eqnarray}
where $\ci_{01}$, $\ci_{11}$ and $\ci_{21}$ are respectively defined by \eqref{eqn.i0}, \eqref{eqn.i1}, \eqref{eqn.i2}. 
\end{lemma}
\begin{remark}
We highlight the fact that, due to the singularity of our equation, both terms $\ca_{00} (t_{\ep}, t_{\tilde{\ep}})$ and $\ci_{01} (t_{\ep}, t_{\tilde{\ep}})$ in \eqref{eqn.a00.lim} are divergent. However, the difference $ \ca_{00} -\ci_{01} $ is a convergent quantity. The same holds true for the limits in \eqref{eqn.a01.lim} and \eqref{eqn.a02.lim}. 
\end{remark}
\begin{proof}[Proof of Lemma \ref{lemma.a00}] 
The proof is an application of  Lemma \ref{lemma.g1g2}.
We only consider the convergence \eqref{eqn.a00.lim}. The convergence in \eqref{eqn.a01.lim} and \eqref{eqn.a02.lim} can be shown in a similar way and will be omitted.
In the following we denote $g=\frac{\partial^{2}\Gamma}{\partial s\partial s'}$. 

For $s\in [0,t]$, set $\eta^{\ep} (s) =  t_{k} $ if $ k \ep\leq s<  (k+1) \ep  $, where we recall that $t_{k} = k\ep$. Then it is easily seen that 
\begin{eqnarray*}
\ca_{00}  (t_{\ep}, t_{\tilde{\ep}})&=& \int_{0}^{t_{\tilde{\ep}}}\int_{0}^{t_{\ep}}  \Gamma \lp\begin{matrix}
\eta^{\tilde{\ep}}( s') -\tilde{\ep}&\eta^{\tilde{\ep}}( s')   \\ \eta^{ {\ep}}( s ) -\ep &  \eta^{ {\ep}}( s )   
\end{matrix}\rp 
  R(t,t)\dd s\dd s' .
\end{eqnarray*}
Then resorting to the mean value theorem for the rectangular increment of $\Gamma$, we get the existence of $\theta =\theta(s,s')$ and $\theta' = \theta'(s,s')$ such that $\theta, \theta' \in [0,1]$ and  
\begin{eqnarray*}
\ca_{00}  (t_{\ep}, t_{\tilde{\ep}})&=& \ep\tilde{\ep}\int_{0}^{t_{\tilde{\ep}}}\int_{0}^{t_{\ep}}  g(\eta(s)-\ep\theta,\eta^{\tilde{\ep}}(s')-\tilde{\ep}\theta') R(t , t )\dd s\dd s' , 
\end{eqnarray*}
where we recall that we have set $g = \frac{\partial^{2} \Gamma}{\partial s\partial s'}$. 
On the other hand, owing to the fact that $\eta(t_{\ep})+\ep$ is the only partition point in $[\tilde{t}, \tilde{t}+\ep)$, we have: 
\begin{eqnarray}\label{a1}
\ci_{01} (t_{\ep}, t_{\tilde{\ep}}) = \ep\tilde{\ep}\Gamma(\eta(t_{\ep}) , \eta(t_{\tilde{\ep}})  ) R(\eta(t_{\ep})+\ep , \eta(t_{\tilde{\ep}})+\tilde{\ep} ) .
\end{eqnarray}
In order to ease notations we denote $\bar{t}^{\ep} = \eta(t_{\ep})$ and $\bar{t}^{\tilde{\ep}} = \eta(t_{\tilde{\ep}})$, so that we can recast \eqref{a1} as
\begin{equation}\label{a2}
\ci_{01} (t_{\ep}, t_{\tilde{\ep}}) = \ep\tilde{\ep} \, \Gamma(\bar{t}^{\ep} , \bar{t}^{\tilde{\ep}}  ) \,
R(\bar{t}^{\ep}+\ep , \bar{t}^{\tilde{\ep}}+\tilde{\ep} ) .
\end{equation}
Introducing the rectangular increment of $\Gamma$ on $[0,\bar{t}^{\ep}]\times[0,\bar{t}^{\tilde{\ep}}]$, we now write \eqref{a2} under the form: 
\begin{eqnarray}\label{eqn.i01}
   \ci_{01} &=&
  \ca_{00}'  (t_{\ep}, t_{\tilde{\ep}}) 
  - \ep\tilde{\ep} \,  R(\bar{t}^{\ep}+\ep, \bar{t}^{\tilde{\ep}}+\tilde{\ep} ) 
  ( \Gamma(0, 0) - \Gamma(0, \bar{t}^{\tilde{\ep}} ) - \Gamma (\bar{t}^{ {\ep}} , 0) ) ,
\end{eqnarray}
where we have set 
\begin{equation*}
\ca_{00}'  (t_{\ep}, t_{\tilde{\ep}})
 = \ep\tilde{\ep} \,  R(\bar{t}^{\ep}+\ep, \bar{t}^{\tilde{\ep}}+\tilde{\ep} ) \Gamma \lp\begin{matrix}
0&\bar{t}^{\tilde{\ep}}  \\ 0&\bar{t}^{ {\ep}} 
\end{matrix}\rp.
\end{equation*}
Furthermore, differentiating the function $\Gamma$ above, $\ca_{00}'  (t_{\ep}, t_{\tilde{\ep}})$ can be expressed as
\begin{eqnarray*}
 \ca_{00}'  (t_{\ep}, t_{\tilde{\ep}})
&=& \ep\tilde{\ep} \, \int_{0}^{\bar{t}^{\tilde{\ep}} } \int_{0}^{\bar{t}^{ {\ep}} } \frac{\partial^{2}\Gamma}{\partial s\partial s'}(s,s') R(\bar{t}^{\ep}+\ep, \bar{t}^{\tilde{\ep}}+\tilde{\ep} ) \dd s\dd s'
\\
&=& \ep\tilde{\ep} \, 
\int_{0}^{\bar{t}^{\tilde{\ep}} } \int_{0}^{\bar{t}^{ {\ep}} } g(s,s') R(\bar{t}^{\ep}+\ep, \bar{t}^{\tilde{\ep}}+\tilde{\ep} ) \dd s\dd s'.
\end{eqnarray*}
A direct application of  Lemma \ref{lemma.g1g2} now yields that  
\begin{eqnarray}\label{eqn.a00a00p}
\lim_{\ep,\tilde{\ep}\to 0+} 
\frac{1}{\ep\tilde{\ep}} \, (\ca_{00}  (t_{\ep}, t_{\tilde{\ep}})- \ca_{00}'  (t_{\ep}, t_{\tilde{\ep}})) = 0.   
\end{eqnarray}
In addition, some easy continuity arguments show that  
\begin{eqnarray*} 
\lim_{\ep,\tilde{\ep}\to 0+} 
R(\bar{t}^{ {\ep}} , \bar{t}^{\tilde{\ep}} ) ( \Gamma(0, 0) - \Gamma(0, \bar{t}^{\tilde{\ep}} ) - \Gamma (\bar{t}^{ {\ep}} , 0) ) 
=
 R(t,t) ( \Gamma(0, 0)-\Gamma(0,t)-\Gamma(t,0) ) .
\end{eqnarray*}
Combining this relation with  \eqref{eqn.i01} and  \eqref{eqn.a00a00p}  ends the proof of \eqref{eqn.a00.lim}.  
\end{proof}



Following is the main result of the paper:
\begin{theorem}\label{thm1}
Let $W$ be the Gaussian field defined in Definition \ref{def.gau} with  time covariance function $R$   and spectral measure $\mu$. The following results on the solution to equation \eqref{eqn.anderson} hold true.

\smallskip
\noindent\underline{Sufficiency}: If $\mu$ satisfies relation   \eqref{hyp.mu} and $R$ is such that 
\begin{eqnarray}\label{eqn.Rupper}
|R(t, t) - R(t, v)-R(u,t)+R(u,v)| \leq K (t-u\wedge v)^{\beta}  
\end{eqnarray}
for a constant $K>0$ and $\beta\in (0, 2]$,  then 

\noindent\emph{(i)} The solution $u(t,x)$ of \eqref{eqn.anderson} exists in the sense of \eqref{eqn.solution} and Definition \ref{def.int}. Namely, for all $t\in \R_{+}$ and $x\in \R^{d}$ the $L^{2}(\Omega)$-convergence in Definition \ref{def.int} holds for $g(s,y)=p_{t-s}(x-y)$. 

\noindent\emph{(ii)} The following identity for the $L^{2}(\Omega)$-norm of the solution holds: 
\begin{eqnarray}\label{eqn.u.L2}
\be[|u(t,x)|^{2}] &=& 
R 
\lp\begin{matrix}
0&t \\ 0&t
\end{matrix}\rp \Gamma(0, 0)
 +\int_{0}^{t }  R 
\lp\begin{matrix}
0&t \\ s'&t
\end{matrix}\rp \frac{\partial \Gamma}{\partial s'}(0, s')\dd s'
 +\int_{0}^{t} R 
\lp\begin{matrix}
s&t \\ 0&t
\end{matrix}\rp \frac{\partial \Gamma}{\partial s}(s, 0)\dd s
\nonumber
\\
&&
+\int_{0}^{t }\int_{0}^{t} R 
\lp\begin{matrix}
s&t \\ s'&t
\end{matrix}\rp \frac{\partial^{2} \Gamma}{\partial s\partial s'}(s, s')\dd s\dd s' .
\end{eqnarray}

\smallskip
 \noindent  \underline{Necessity}: Suppose that the solution of \eqref{eqn.anderson} exists in the sense of \eqref{eqn.solution}, that   identity \eqref{eqn.u.L2} holds true, and that $R$ is such that 
  \begin{eqnarray}\label{eqn.Rlower}
|R(t, t) - R(t, v)-R(u,t)+R(u,v)| \geq K (t-u\wedge v)^{\beta}  
\end{eqnarray}
for some $\beta\in (0, 2]$. 
 Then  the spectral measure $\mu$ satisfies relation   \eqref{hyp.mu}. 
\end{theorem}
\begin{proof}  
Recall that in Lemma \ref{lemma.ue.L2} we have shown that 
\begin{eqnarray*}
&&\be(u_{\ep}(t, x)u_{\tilde{\ep}}(t, x) ) 
\\
& &\quad\quad=
(\ep\tilde{\ep})^{-1} \lp
 \ca_{0}(t_{\ep}, t_{\tilde{\ep}}) +\ci_{0}(t_{\ep}, t_{\tilde{\ep}})+\ci_{1}(t_{\ep}, t_{\tilde{\ep}})+\ci_{2}(t_{\ep}, t_{\tilde{\ep}})+\ci_{3}(t_{\ep}, t_{\tilde{\ep}})+\ci_{4}(t_{\ep}, t_{\tilde{\ep}}) 
 \rp.
\end{eqnarray*}
Moreover, going back to the definition  \eqref{eqn.a0} of $\ca_{0}$ and \eqref{eqn.a0.tilde} of $\tilde{\ca}_{0}^{\ep, \tilde{\ep}}$, some elementary manipulations show that
\begin{eqnarray*}
\ca_{0} (t_{\ep}, t_{\tilde{\ep}}) &=& 
\tilde{\ca}_{0}^{\ep, \tilde{\ep}}  (t_{\ep}, t_{\tilde{\ep}}) - \ca_{00} (t_{\ep}, t_{\tilde{\ep}}) + \ca_{00} (t_{\ep}, t_{\tilde{\ep}}) + \ca_{00} (t_{\ep}, t_{\tilde{\ep}}),
\end{eqnarray*}
 where $\ca_{00}$, $\ca_{01}$ and $\ca_{02}$ are introduced in Lemma \ref{lemma.a00}. Thanks to the decompositions \eqref{eqn.i0}, \eqref{eqn.i1} and \eqref{eqn.i2}, we thus end up with 
\begin{align*}
&
\be(u_{\ep}(t, x)u_{\tilde{\ep}}(t, x)) 
\\
&\quad= 
(\ep\tilde{\ep})^{-1}
\Big(
\tilde{\ca}_{0}^{\ep, \tilde{\ep}} (t_{\ep}, t_{\tilde{\ep}})  +\ci_{3}(t_{\ep}, t_{\tilde{\ep}})+\ci_{4}(t_{\ep}, t_{\tilde{\ep}})-\ci_{00}(t_{\ep}, t_{\tilde{\ep}})-\ci_{10}(t_{\ep}, t_{\tilde{\ep}})-\ci_{20}(t_{\ep}, t_{\tilde{\ep}})  
  \\
  &
  \quad\quad\quad
   + (\ci_{11}(t_{\ep}, t_{\tilde{\ep}}) +\ca_{02}(t_{\ep}, t_{\tilde{\ep}}))+ (\ci_{21}(t_{\ep}, t_{\tilde{\ep}}) +\ca_{01}(t_{\ep}, t_{\tilde{\ep}}))
  +(  \ci_{01}(t_{\ep}, t_{\tilde{\ep}}) -\ca_{00}(t_{\ep}, t_{\tilde{\ep}}) ) 
  \Big).
\end{align*}
Now we can apply  Lemma \ref{lemma.i3i4} to $\ci_{3}$ and $\ci_{4}$, Lemma \ref{lemma.a0} to $\tilde{\ca}_{0}^{\ep, \tilde{\ep}}$, Lemma \ref{lemma.a00} to $ (\ci_{11}  +\ca_{02} )+ (\ci_{21}  +\ca_{01} )
  +(  \ci_{01}  -\ca_{00}  )$, and apply the continuity of $R$ and $\Gamma$ to $\ci_{00}-\ci_{10}-\ci_{20}$, which shows the convergence of $\be(u_{\ep}(t, x)u_{\tilde{\ep}}(t, x) ) $ as $\ep, \tilde{\ep}\to 0+$ and thus    completes the proof of (i). Item (ii) can be shown immediately by rearranging the limits of these convergences. Notice that in all the aforementioned Lemma \ref{lemma.i3i4},  \ref{lemma.a0} and \ref{lemma.a00} we  are only using the upper bound \eqref{eqn.Rupper} on the increments of $R$ instead of relation \eqref{hyp.R}. 
  
  To prove   the necessity, we apply  Fubini's theorem to the last term on the right-hand side of \eqref{eqn.u.L2} and then invoke relation \eqref{eqn.ddk}. This yields  the inequality
\begin{eqnarray*}
  \int_{\R^{d}} {|\xi|^{4}} \mu(\dd \xi) \int_{0}^{t}\int_{0}^{t}     e^{-\frac{(2t-s-s')|\xi|^{2}}{2}}  R \lp\begin{matrix}
s&t \\ s' &t  
\end{matrix}\rp
    \dd s\dd s' <\infty. 
\end{eqnarray*} 
Taking into account that $R$ satisfies relation \eqref{eqn.Rlower} and using some elementary change of variable formulas, we conclude that the spatial spectral measure $\mu$ satisfies relation  \eqref{hyp.mu}. 
  \end{proof}

\subsection{Examples of application}
Let us now give some  examples of covariance functions satisfying our Hypothesis \eqref{hyp.R} and \eqref{hyp.mu}. We will also compare our Theorem \ref{thm1} with the numerous results available for equation \eqref{eqn.anderson} and  for the parabolic Anderson model. Recall that the parabolic Anderson model is given by:
\begin{equation} \label{eqn.multiplicative_anderson}
\begin{cases}
&\partial_{t} u    = \frac12 \Delta u   + u\dw , \qquad t\in\ott,\, x\in\R^{d} \\
&u(0, x)   =u_{0}(x),
\end{cases}
\end{equation}
where $u_{0}$ is a given smooth and non degenerate initial condition. In \eqref{eqn.multiplicative_anderson},
the product between $u$ and $\dw $ is understood in the  Wick sense. 

\begin{example} 
Assume that the noise is  white in time, namely,  
$R(s,t)=s\wedge t$. In this case  $\beta=1$ and  our condition \eqref{hyp.mu} recovers Dalang's condition (see \cite{Da}): 
$$
\int_{\R^{d}} \frac{1}{1+|\xi|^{2}}\mu(d\xi)<\infty.
$$ 
In particular, if    the spatial covariance is given by a Riesz kernel
$\laa(x)=|x|^{-\eta}$ for some $\eta>0$:
\[
\be(\dot W(t,x)\dot W(s,y))=\delta(t-s)\Lambda(x-y)\,,\]
then
$\mu(d\xi)=c_{\eta,d}|\xi|^{-(d-\eta)} \, d\xi$. 
Condition \eqref{hyp.mu} is thus equivalent to  $\eta<2$, as in \cite{Da}.  

Still in the Brownian case in time, suppose that $d=1$ and that the spatial 
noise is   fractional   with Hurst parameter $H\in (0, 1)$,  which is equivalent to   
$\mu(d\xi)=C_H|\xi|^{1-2H}$. The condition~\eqref{hyp.mu}  is thus satisfied 
for all $H\in (0, 1)$.  
This is in sharp contrast with the multiplicative case of equation \eqref{eqn.multiplicative_anderson}. Indeed, in the multiplicative case with $d=1$, one has to assume $H>1/4$ in order to get existence and uniqueness of a function valued solution (see e.g. \cite{HHLNT}).

Consider now a Brownian motion in time whose spatial  covariance is  given by  the  Bessel kernel 
\[
\laa(x)=\int_0^{\infty}w^{\frac{\eta-d}{2}}e^{-w}e^{-\frac{|x|^2}{4w}}dw\,.
\]
Then we have
   $\mu(d\xi)=c_{\eta,d} (1+|\xi|^2)^{-\frac{\eta}{2}} d\xi$ and \eqref{hyp.mu}  is satisfied  if and only if  $\eta> d-2$. This result is implicitly contained in \cite{Da}.
\end{example}

\begin{example}   
Assume that the noise is fractional in time with Hurst parameter $H_0\in (0, 1)$,  namely, 
 $R(s,t)=  \frac{1}{2}\lp |s|^{2H_0}+|t|^{2H_0}-|s-t|^{2H_0}\rp$.
In this situation $\beta=2H_0$ in \eqref{hyp.R} and \eqref{hyp.mu}. 

The particular case    $H_{0}>1/2$     has been considered in \cite{BT}    for equation~\eqref{eqn.anderson}, where some necessary and sufficient conditions on the covariance function in space have been obtained for the existence of the solution.   

Suppose now that  the noise is also fractional in space with Hurst parameters $(H_1, 
\cdots, H_d)$,  which means that  the spatial  fractional covariance is given on $\R^{2d}$ by 
\begin{equation*}
\prod_{i=1}^{d} R_{H_{i}}(x_{i}, y_{i}),
\quad\text{where}\quad
R_{H_{i}}(u,v)= \frac{1}{2}    (|u|^{2H_i }+|v|^{2H_i }-|u-v|^{2H_{i}}).
\end{equation*}
Then we have $\mu(d\xi)=C_H
 \prod_{i=1}^d |\xi_i|^{1-2H_i}d\xi$, where $C_H$ is a constant
depending on the parameters $H_i$. In this situation the condition  \eqref{hyp.mu} becomes 
\[
\int_{\R^d} \frac1{1+|\xi|^{4H_0}}\prod_{i=1}^d |\xi_i|^{1-2H_i}d\xi<\infty\, ,
\]
and an easy calculation shows
that this is equivalent to  
\begin{equation} 
2H_0+\sum_{i=1}^d H_i > d \,.\label{e.3.53}
\end{equation}
This condition has to be compared to what is obtained for multiplicative equations like~\eqref{eqn.multiplicative_anderson}. In this context, \cite[Example 2.6]{HHNT} asserts the existence of an  $L^2$ solution $u$ under the condition $H_i>1/2$ for all $i=0,1,\dots, d$ and  the additional lower bound $\sum_{i=1}^d H_i>d-1$, which  is a stronger assumption than \eqref{e.3.53}. In a recent paper by Chen \cite{C2}, this condition has been improved to rough  cases with some of the $H_{1},\dots, H_{d}$ less than $1/2$. For example,  when $d=1$ it has been shown that  \eqref{eqn.multiplicative_anderson} admits a unique solution as soon as $H_{0}>1/2$ and $H_{0}+H_{1}>3/4$.

 Let us particularize condition  \eqref{e.3.53} to a white noise in space, that is $H_1=\cdots=H_d=1/2$. In this case, equation \eqref{e.3.53} becomes $H_0>d/4$.  We can compare this result to two situations studied in \cite{HN}:
 \begin{itemize}
\item 
When $d=1$, our condition reads as $H_0>1/4$, while \cite{HN} was assuming $H_0>1/2$ in the multiplicative case \eqref{eqn.multiplicative_anderson}.  

\item
When $d=2$, equation \eqref{e.3.53} becomes $H_0>1/2$, while only the existence for small time for \eqref{eqn.multiplicative_anderson} was established in \cite{HN} under the condition $H_0>1/2$ and $H_1=H_2=1/2$.  
\end{itemize}

Let us  also    mention a recent result  in \cite{Deya}, which considered a Stratonovich-type  nonlinear heat equation  with fractional noise in time and space and with $d=1$ in space. It has been shown in \cite{Deya} that in this case the   solution   exists when $2H_{0}+H_{1}>2$, while a renormalization of the system is required to solve the equation when $2\geq 2H_{0}+H_{1}>5/3$. 
\end{example}

\begin{example}
  Assume that the noise is independent of the time parameter $t$. In this case, $R(s,t)=st$ which means that  $\beta=2$.  The condition  \eqref{hyp.mu} becomes 
\begin{equation}\label{eq:cdt-space-noise}
\int_{\R^{d}} \frac{1}{1+|\xi|^{4}}  \mu (\dd \xi)<\infty \,. 
\end{equation}
Condition \eqref{eq:cdt-space-noise} can  be compared again to the multiplicative case \eqref{eqn.multiplicative_anderson}. Namely, we can quote \cite[Theorem 3.9]{HHNT}, where Dalang's condition $\int_{\R^{d}} \frac{1}{1+|\xi|^{2}}\mu(d\xi)<\infty$  had to be assumed in order to solve  equation \eqref{eqn.multiplicative_anderson}.
\end{example}

 \section{Solving the heat equation in a   Besov space}\label{section.solution2}

This section sheds a different light on the existence problem for equation \eqref{eqn.anderson}. Namely, we will now consider the noise $ \partial W :=  \frac{\partial^{d} W}{\partial x_{1} \cdots \partial x_{d} }     $ as a distribution in a certain Soblev space of negative order. Then we will quantify how the heat flow regularizes $W$ in order to give a meaning to $u$ as a function. This analysis obviously requires some preliminary  background about Littlewood-Paley theory (recalled below) and yields  some slightly non optimal results. However, let us mention that the computations leading to the existence  of a solution are simpler within this framework than in the previous section. Furthermore, the Besov space method also brings some regularity results for the solution $u$ at no additional cost.

We briefly recall some elements of the Besov space theory. The readers are referred to  \cite[Chapter 2]{BCD} for further details, and to \cite{MW} for an analysis of Besov spaces with weights.  We  first give a  result which provides us with the dyadic partition of unity (see a more complete statement in \cite[Proposition 2.10]{BCD}): 
\begin{proposition}\label{prop.little}
Let $\cac$ be the annulus $\{ \xi\in \R^{d}: 3/4 \leq |\xi|\leq 8/3 \}$. Then there exist radial functions $\chi$ and $\varphi$, valued in the interval $[0,1]$, belonging respectively to $\cd (B(0, 4/3))$ and $D(\cac)$, and such that
\begin{eqnarray*}
\forall \xi \in \R^{d}, \quad \chi(\xi) + \sum_{j\geq 0} \varphi(2^{-j} \xi) = 1,
\quad \text{and}\quad
\forall \xi \in \R^{d} \setminus \{0\} , \quad   \sum_{j \in \Z} \varphi(2^{-j} \xi) = 1,
\end{eqnarray*}
The following support type conditions are also satisfied by $\vp$ and $\chi$:
\begin{eqnarray*}
 |j-j'|\geq 2 \Rightarrow \emph{Supp} \varphi (2^{-j} \cdot) \cap \emph{Supp} \varphi (2^{-j'} \cdot) =\emptyset,
\\
j\geq 1 \Rightarrow \emph{Supp} \chi \cap \emph{Supp} \varphi (2^{-j} \cdot) =\emptyset.
\end{eqnarray*}
 \end{proposition}
Let us now define the dyadic blocks $\Delta_{j} $, which are the basic bricks of Littlewood-Paley's analysis. 
\begin{definition}\label{def.blocks}
Let $\chi$ and $\varphi$ be the two functions constructed in Proposition \ref{prop.little} and write 
$h= \cf^{-1}\varphi$ and $\tilde{h} = \cf^{-1} \chi$. The \emph{nonhomogeneous dyadic blocks $\Delta_{j}$} are defined by
\begin{eqnarray*}
\Delta_{j} u = 0 \quad \text{if}\quad j\leq -2, \qquad \Delta_{-1} u   = \int_{\R^{d}} \tilde{h}(y) u(x-y) dy, 
\end{eqnarray*}
and for $j\geq 0$, we have
\begin{eqnarray}\label{eqn.blocks}
    \Delta_{j} u   = 2^{jd} \int_{\R^{d}} h(2^{j} y) u(x-y)dy . 
 \end{eqnarray}
 
\end{definition}

With the previous notations in hand, we shall now define a family of Besov type spaces in which our noise $\partial W$ will sit.  
\begin{definition}
Let $\kappa\in \R$ and $1\leq q< \infty$. 
We will consider a spatial weight  $\rho_{\si} (x)= \frac{1}{(1+|x|)^{\si}} $ defined for $x\in \R^{d} $ and $\si>d$.  The non-homogeneous weighted Besov space $B^{\kappa}_{q}$ consists of all tempered distributions $u$ such that
  \begin{eqnarray}\label{eqn.besovnorm}
\| u \|_{B^{\kappa}_{ q}}^{2q} : = 
 \sum_{j\in \Z}   2^{2qj\kappa} \|\Delta_{j} u \|_{L^{2q}_{\rho_{\si}}}^{2q}  <\infty   .
\end{eqnarray}
Here $L^{2q}_{\rho} $ denotes the space  $L^{2q}( \R^{d}, \rho(x)dx )$. 
\end{definition}

\begin{remark}\label{remark.holder}
We have used the norms defined by \eqref{eqn.besovnorm} for computational sake. However, we should mention that if $f \in B^{\kappa}_{q}$ for $\kappa>\frac{d}{2q}$ and $\phi \in \cd (\R^{d})$, then $f\phi \in C^{\al}$ for $\al=\kappa-\frac{d}{2q}$. We are thus able to embed locally our Besov spaces $B^{\kappa}_{q}$ into H\"older type spaces. 
\end{remark}
\begin{proof}[Proof of Remark \ref{remark.holder}]
If $f \in B^{\kappa}_{q}$, it is shown in \cite[Proposition 3.27]{MW} that $f\phi$ belongs to the non weighted Besov space $B^{\kappa}_{2q, 2q}$. Then owing to \cite[Proposition 2.71]{BCD} we have that $B^{\kappa}_{2q, 2q}$ is continuously embedded into the H\"older space $C^{\al}$ with $\al =\kappa-\frac{d}{2q}$.  
\end{proof}

For our convenience we are working in this section with a noise $\partial W$, which has to be thought of as an integrated version in time of the noise $\dot{W}$ which appears in equation \eqref{eqn.anderson}. We now define this Gaussian family more rigorously. 
 
 \begin{definition}
 Let $W$ be the Gaussian family introduced in Definition \ref{def.gau}. We define another centered Gaussian family $\partial W = \{ \partial W_{st} (\varphi): 0\leq s\leq t\leq \tau, \varphi \in \cs (\R^{d}) \}$ by 
 \begin{eqnarray*}
\partial W_{st} (\varphi) = W(\mathbf{1}_{[s, t]}\otimes \varphi). 
\end{eqnarray*}
 \end{definition}
 
 We now state the assumptions on the covariance of our noise in a slightly different way with respect to \eqref{hyp.R} and \eqref{hyp.mu}. 
 
 \begin{hypothesis}\label{hyp.Rmu}
 We assume that there exists   $\beta  >0$ and $\beta'\in (0, 2]$ such that  the function $R$ and the measure $\mu$ appearing in \eqref{eqn.cov.w} satisfy the following conditions:
 \begin{eqnarray*}
|R(t, t) -R(u, t)-R(t, v)+R(u, v)| \leq K |t-(u\wedge v)|^{\beta'} \quad \text{and} \quad \int_{\R^{d}} \frac{1}{(1+|\xi|)^{2\beta}} \mu (d\xi) <\infty.
\end{eqnarray*}

 \end{hypothesis}
 
 In the following we prove that $\partial W$ can also be seen as a H\"older continuous function of time taking values in a weighted Besov space. 
\begin{lemma}\label{lemma.dW}
Let $W$ be the  Gaussian field in Definition \ref{def.gau} with time covariance $R$ and spatial spectral measure $\mu$, and suppose that  Hypothesis \ref{hyp.Rmu} holds true. Then $\partial W_{st} \in B^{\kappa}_{ q}$ for all $\kappa<-\beta$, $q\geq 1$, and $0\leq s \leq t$. 
Moreover, for all $\epsilon>0$, $q\geq 1$ and $\kappa<-\beta$ there exists a random variable $Z$ admitting moments of all orders such that for all $0\leq s<t\leq \tau$ we have
\begin{eqnarray}\label{eqn.pdW}
\|\partial W_{st}\|_{B^{\kappa}_{q}} \leq Z(t-s)^{\frac{\beta'}{2}-\epsilon}. 
\end{eqnarray}
\end{lemma}
\begin{proof}
 By Definition \ref{def.blocks} of $\Delta_{j}$ we first write  
\begin{eqnarray*}
\Delta_{j} \partial W_{st}    = W(\mathbf{1}_{[s,t]} \otimes 2^{jd} h(2^{j}(x-\cdot))) .  
\end{eqnarray*}
  Since $  \Delta_{j} \partial W_{st} $ is Gaussian, by the hypercontractivity property we obtain
\begin{eqnarray}\label{eqn.bdLq}
\be \lc \left\|  \Delta_{j} \partial W_{st}  \right\|_{L_{\rho_{\si}}^{2q}}^{2q} \rc &=& \be \lc \int  |  \Delta_{j} \partial W_{st}(x)  |^{2q} \rho_{\si} (x)dx \rc
\nonumber
\\
&\leq& c_{q}   \lc \int  \be [|  \Delta_{j} \partial W_{st}(x)  |^{2}]^{ q} \rho_{\si} (x)dx \rc.
\end{eqnarray}
Notice that the function $2^{jd} h(2^{j}y) $ appearing in formula \eqref{eqn.blocks} satisfies
\begin{eqnarray*}
\cf[ 2^{jd} h(2^{j}\cdot) ](\xi) = \varphi (2^{-j} \xi). 
\end{eqnarray*}
Therefore, 
resorting to \eqref{eqn.cov.w} we can write  
\begin{eqnarray*}
\be [|  \Delta_{j} \partial W_{st}(x)  |^{2}] = R 
\lp\begin{matrix}
s&t\\ s&t
\end{matrix}\rp \int_{\R^{d}}   (1+|\xi|)^{2\beta}|\varphi (2^{-j}\xi)|^{2} \frac{\mu(d\xi)}{ (1+|\xi|)^{2\beta}}. 
\end{eqnarray*}
Hence,  taking into account the fact that  $ (1+ 2^{-j} |\xi|)^{2\beta}|\varphi (2^{-j}\xi)|^{2}  $  is uniformly bounded in $\xi$ and Hypothesis \ref{hyp.Rmu} on $\mu$,  we obtain
\begin{eqnarray}\label{eqn.bdL2}
\be [|  \Delta_{j} \partial W_{st}(x)  |^{2}]  &\leq & K2^{ 2j\beta} (R(t, t) - 2R(s,t)+R(s,s)) .
\end{eqnarray}
Now invoking \eqref{eqn.besovnorm}, and then substituting \eqref{eqn.bdL2} into \eqref{eqn.bdLq} and taking into account  the fact that  $\kappa <-\beta$, we  have 
\begin{eqnarray}\label{eqn.BesovW}
\be[\|\partial W_{st}\|^{2q}_{B^{\kappa}_{ q}}] 
= 
 \be\lc\sum_{j\in \Z}   2^{2qj\kappa} \|\Delta_{j} \partial W_{st} \|_{L^{2q}_{\rho_{\si}}}^{2q} \rc
\leq K (R(t, t) - 2R(s,t)+R(s,s)) ^{q}. 
\end{eqnarray} 
Once \eqref{eqn.BesovW} is proved, our assertion \eqref{eqn.pdW} is shown thanks to Hypothesis \ref{hyp.Rmu} on $R$ and a standard application of Garsia's   lemma. 
 \end{proof}
 
 In order to transfer our Lemma \ref{lemma.dW} on $\partial W$ to properties of the heat equation, we recall some results  (see  \cite{MW}) about the smoothing effects of the heat flow $p_{t}: =e^{t\Delta }$ in Besov spaces.  
\begin{lemma}\label{lemma.smoothing1}
Consider $\al \in \R$ and two real numbers $\eta, \kappa$ such that $\eta\geq \al$, $\kappa\geq \al$ and $\kappa-\al\leq 2$. 
Let    $q\geq 1  $ be a real number. Then there  exists a constant $K<\infty$ such that uniformly over $t>0$ we have
\begin{eqnarray*}
\| p_{t}  f \|_{B^{\eta}_{ q}} &\leq& \frac{K}{ t^{\frac{\eta-\al}{2}}} \|f\|_{B^{\al}_{ q}}   
\quad\quad\text{and}\quad\quad
\| (\id - p_{t}) f \|_{B^{\al}_{ q}} \leq K t^{\frac{\kappa-\al}{2}} \|f\|_{B^{\kappa}_{ q}}.  
\end{eqnarray*}

\end{lemma}

Recall that in Section \ref{section.noise} (see Definition \ref{def.int}) the Wiener integrals were considered as $L^{2}(\Omega)$ limit of Riemann sums. In this section we introduce the same kind of regularization, whereas the limits are considered in an almost sure sense. 

\begin{definition}\label{def.int2}
Consider the dyadic partition of $\R_{+}$ defined by   $t^{n}_{k}=2^{-n}k$, for $k\in \N$  and $ n\in \N $. The solution of  equation \eqref{eqn.anderson} is defined as the almost sure limit of
\begin{eqnarray}\label{eqn.int2}
u^{n}_{t}=\sum_{0\leq t^{n}_{k}< t} p_{t-t^{n}_{k}}   \partial W_{t^{n}_{k}t^{n}_{k+1}} ,
\end{eqnarray}
whenever   $u^{n}_{t}$ converges 
in some Besov space $B^{\eta}_{ q}$ with $\eta>0$ and $q\geq 1$.  In relation \eqref{eqn.int2}, $p_{s} \partial W_{uv}$ has to be understood as the action of the semigroup $p_{s}$ on the distribution $\partial W_{uv} \in B^{\kappa}_{q}$, as introduced in Lemma \ref{lemma.smoothing1}. 
\end{definition}

We are now ready to solve equation \eqref{eqn.anderson} within our Besov space framework. As mentioned above, this method brings out regularity results on the solution in a natural way. 
\begin{theorem}\label{thm2}
Let $W$ be the centered Gaussian field introduced in   Definition \ref{def.gau} with   spatial spectral measure $\mu$
and time covariance $R$. We suppose  that  Hypothesis \ref{hyp.Rmu} holds true for some constants $\beta>0$ and $\beta'\in (0, 2]$ such that $\beta'>\beta$. We consider an arbitrarily large time horizon $\tau>0$. Then the following holds true for our stochastic heat equation.

\noindent
\emph{(i)}   
Equation \eqref{eqn.anderson} admits a random field solution $\{u(t, x); \, t\in [0,\tau] , x\in \R^{d}\}$ in the sense of Definition \ref{def.int2}.

\noindent
\emph{(ii)}   
In addition, for any $\epsilon,\eta>0$ such that $\eta+\epsilon <\beta'-\beta$ and for all $q\geq 1$, the solution $u$ to \eqref{eqn.anderson} almost surely sits in the space  $C^{\frac{\beta'-\beta-\eta}{2}-\epsilon}([0 , \tau]; B^{\eta}_{ q})$. 
\end{theorem}
\begin{proof}
Let $0\leq s<t\leq \tau$ and recall that the dyadic partition of $\R_{+}$ is defined by $t^{n}_{k}=k2^{-n}$ for $k,n\ge 0$. 
With some elementary computations we easily get 
\begin{eqnarray}\label{eqn.un}
&&
 (u^{n}_{t}  -u^{n}_{s})-(u^{n+1}_{t} -u^{n+1}_{s}  )
 \nonumber
\\
& &=
 \sum_{s\leq t^{n+1}_{2k}< t} (p_{t-t^{n+1}_{2k}} - p_{t-t^{n+1}_{2k+1}} )  \partial W_{t^{n+1}_{2k+1}t^{n+1}_{2k+2}}
 \nonumber
 \\
 &&\quad+ 
 \sum_{0\leq t^{n+1}_{2k}< s} [(p_{t-t^{n+1}_{2k}}  -p_{s-t^{n+1}_{2k}} ) - (p_{t-t^{n+1}_{2k+1}}  -p_{s-t^{n+1}_{2k+1}} ) ] \partial W_{t^{n+1}_{2k+1}t^{n+1}_{2k+2 }}
 \nonumber
\\
&& = \cu_{1}+\cu_{2} ,
\end{eqnarray}
where we have set
\begin{eqnarray*}
&& \cu_{1}=
 \sum_{s\leq t^{n+1}_{2k}< t} p_{t-t^{n+1}_{2k+1 }}   
 \lc
  (p_{2^{-(1+n)}} - \id)  \partial W_{t^{n+1}_{2k+1}t^{n+1}_{2k+2}}
\rc
\nonumber
 \\
&&\cu_{2}= 
 \sum_{0\leq t^{n+1}_{2k}< s} (p_{t-s}-\id)p_{s-t^{n+1}_{2k+1}}  (p_{2^{-(1+n)}}-\id) \partial W_{t^{n+1}_{2k+1}t^{n+1}_{2k+2 }}.
\nonumber
\end{eqnarray*}
In the following, we  bound $\cu_{1}$ and $\cu_{2}$ separately, starting with $\cu_{1}$. 

In order to analyze the term $\cu_{1}$ in \eqref{eqn.un}, we first tune the parameters of our Besov space. Namely we consider three parameters $\kappa$, $\al$, $\eta$ satisfying the following conditions:
\begin{eqnarray}\label{eqn.para}
\kappa=-\beta-\epsilon, \quad \al< \kappa+\beta'-2-2\epsilon , \quad 0<\eta<2+\al
\end{eqnarray}
for an arbitrary small constant $\epsilon >0$. 
 Notice that for such values of the parameters we have 
 \begin{eqnarray*}
\frac{\kappa-\al}{2} + \frac{\beta'}{2}-\epsilon-1>0 \quad \text{and} \quad  \frac{\al - \eta}{2}>-1. 
\end{eqnarray*}
  Also observe that this choice of $\al$ and $\eta$ is possible as long as $\eta<\beta'-\beta-3\epsilon$. 
 With these values of $\kappa$, $\al$ and $\eta$,   we now apply  Lemma \ref{lemma.smoothing1}   and our estimate \eqref{eqn.pdW}  to the terms of $\cu_{1}$. This yields   the following estimate 
\begin{eqnarray}\label{eqn.ppW}
&&
\left\|
  p_{t-t^{n+1}_{2k+1 }}   
 \lc
  (p_{2^{-(1+n)}} - \id)  \partial W_{t^{n+1}_{2k+1}t^{n+1}_{2k+2}}
\rc
\right\|_{B^{\eta}_{ q}}
\nonumber
\\
&&\leq Z (t-t^{n+1}_{2k+1 })^{\frac{\al - \eta}{2}} \left\|
 (p_{2^{-(1+n)}} - \id)  \partial W_{t^{n+1}_{2k+1}t^{n+1}_{2k+2}}
\right\|_{B^{\al}_{ q}}
\nonumber
\\
&&\leq Z (t-t^{n+1}_{2k+1 })^{\frac{\al - \eta}{2}} 
 (2^{-n})^{\frac{\kappa-\al}{2} + \frac{\beta'}{2}-\epsilon}.
\end{eqnarray} 
Summing \eqref{eqn.ppW} over $k$   and taking into account \eqref{eqn.un} we obtain an upper-bound for $\cu_{1}$:
\begin{eqnarray}\label{eqn.u1}
\|\cu_{1}\|_{B^{\eta}_{q}} \leq 
\sum_{s\leq t^{n+1}_{2k}< t}\left\|
  p_{t-t^{n+1}_{2k+1 }}   
 \lc
  (p_{2^{-(1+n)}} - \id)  \partial W_{t^{n+1}_{2k+1}t^{n+1}_{2k+2}}
\rc
\right\|_{B^{\eta}_{ q}}
\nonumber
\\
\leq
Z (2^{-n})^{\frac{\kappa-\al}{2} + \frac{\beta'}{2}-1-\epsilon} (t-s)^{\frac{\al-\eta}{2}+1} .
\end{eqnarray}

The term $\cu_{2}$ is bounded along the same lines as $\cu_{1}$.  Namely, we take $\eta'$ such that $\al>\eta'-2>\eta-2$.  Applying Lemma \ref{lemma.smoothing1} to $\cu_{2}$ and then following the estimates of \eqref{eqn.ppW}, we are left with  
\begin{eqnarray}\label{eqn.u2}
\|\cu_{2}\|_{B^{\eta}_{q}} \leq  \sum_{0\leq t^{n+1}_{2k}<s}Z (t-s)^{\frac{\eta'-\eta }{2}} (s-t^{n+1}_{2k+1 })^{\frac{\al - \eta'}{2}} 
 (2^{-n})^{\frac{\kappa-\al}{2} + \frac{\beta'}{2}-\epsilon}
\nonumber \\
 \leq Z(t-s)^{\frac{\eta'-\eta }{2}} (2^{-n})^{\frac{\kappa-\al}{2} + \frac{\beta'}{2}-1-\epsilon} s^{\frac{\al - \eta'}{2}+1} .
\end{eqnarray}

We now plug inequalities   \eqref{eqn.u1} and \eqref{eqn.u2} into \eqref{eqn.un}, which yields  
\begin{eqnarray}\label{eqn.du}
\|(u^{n}_{t}  -u^{n}_{s})-(u^{n+1}_{t} -u^{n+1}_{s}  )\|_{B^{\eta}_{ q}}   &\leq & Z(t-s)^{\frac{\eta'-\eta }{2}} (2^{-n})^{\frac{\kappa-\al}{2} + \frac{\beta'}{2}-1-\epsilon}   .
\end{eqnarray}
Note that a similar estimate for  $u^{0}_{t} - u^{0}_{s}$ also holds:
\begin{eqnarray}\label{eqn.u0}
\| u^{0}_{t}  -u^{0}_{s} \|_{B^{\eta}_{ q}}   &\leq & Z(t-s)^{\frac{\eta' -\eta }{2} }    ,
\end{eqnarray}
which can be shown in a similar way as for \eqref{eqn.du}.

It now follows easily  from \eqref{eqn.du} and \eqref{eqn.u0} and the inequality $\frac{\kappa-\al}{2} + \frac{\beta'}{2}-\epsilon-1>0 $ that $u^{n}_{t} $ converges to $u_{t} $ in ${B^{\eta}_{ q}}$ and that
\begin{eqnarray}\label{eqn.uho}
\sup_{s, t\in [0, \tau]} \frac{ \| u_{t} -u_{s} \|_{B^{\eta}_{q}} }{|t-s|^{\frac{\eta'-\eta}{2}}} \leq Z.
\end{eqnarray}
Let us now compute the order of magnitude of $\eta'-\eta$. Recall that we had to impose 
\begin{eqnarray*}
\eta<\eta'<\al+2.
\end{eqnarray*}
In addition, according to \eqref{eqn.para} we have $\al<\beta'-\beta-2-3\epsilon$. Hence it is readily checked that $\eta'$ is at most $(\beta'-\beta-3\epsilon)-$. In conclusion, the exponent $\eta'-\eta$ in \eqref{eqn.uho} can take any value in $(0, \beta'-\beta)$, and we get   $u \in C^{\frac{\beta'-\beta-\eta}{2}-3\epsilon}([0 , \tau]; B^{\eta}_{ q})$. The proof is now complete. 
 \end{proof}

\begin{remark}
Combining Theorem \ref{thm2} , Remark \ref{remark.holder} and due to the fact that we can consider an arbitrarily large number $q$ in Theorem \ref{thm2}, we get that the random field solution $u$ to equation \eqref{eqn.anderson} is a $[\frac{1}{2}(\beta'-\beta-\eta)-\epsilon, \eta]$-H\"older function on $[0, \tau]\times [-M, M]^{d}$ for $\epsilon$ arbitrarily small, $M$ arbitrarily large and any $\eta\in (0, \beta'-\beta)$. 
\end{remark}

 
 
 


\end{document}